\documentclass[11pt,a4paper,twoside,bibtotoc,final]{scrartcl}

\usepackage{a4wide}
\usepackage{amsfonts}
\usepackage{amsmath}
\usepackage{amssymb}
\usepackage{amsthm}
\usepackage[numbers]{natbib}
\usepackage{booktabs}
\usepackage{algorithm,algpseudocode}
\usepackage{pgfplots}
\pgfplotsset{compat=newest}
\usepackage{graphicx}
\usepackage{color}
\usepackage{colortbl}
\usepackage{multirow}
\usepackage{boxedminipage}
\usepackage{enumerate}
\usepackage[caption=false,font=footnotesize]{subfig}
\usepackage{rotating}
\usepackage[mathscr]{eucal}
\usepackage{hyperref}
\usepackage[inline]{enumitem}

\usepackage{todonotes}

\newcommand{\ChebPol}{\ensuremath{\textnormal{C}\Pi}}
\newcommand{\N}{\ensuremath{\mathbb{N}}}

\newcommand{\T}{\ensuremath{\mathbb{T}}}

\newcommand{\Z}{\ensuremath{\mathbb{Z}}}

\newcommand{\R}{\ensuremath{\mathbb{R}}}
\newcommand{\C}{\ensuremath{\mathbb{C}}}

\newcommand{\ii}{\textnormal{i}}
\newcommand{\e}{\textnormal{e}}

\newcommand{\ceil}[1]{\left\lceil#1\right\rceil}

\newcommand{\zb}[1]{\ensuremath{\boldsymbol{#1}}}

\renewcommand{\ln}{\mathrm{ln\,}}

\newcommand{\boldk}{{\ensuremath{\boldsymbol{k}}}}
\newcommand{\boldj}{{\ensuremath{\boldsymbol{j}}}}
\newcommand{\boldl}{{\ensuremath{\boldsymbol{l}}}}
\newcommand{\boldt}{{\ensuremath{\boldsymbol{t}}}}
\newcommand{\boldh}{{\ensuremath{\boldsymbol{h}}}}

\newcommand{\boldx}{{\ensuremath{\boldsymbol{x}}}}

\newcommand{\boldz}{{\ensuremath{\boldsymbol{z}}}}

\newcommand{\bolda}{{\ensuremath{\boldsymbol{a}}}}
\newcommand{\boldb}{{\ensuremath{\boldsymbol{b}}}}
\newcommand{\boldP}{{\ensuremath{\boldsymbol{P}}}}

\newcommand{\boldc}{{\ensuremath{\boldsymbol{c}}}}

\newcommand{\boldzero}{{\ensuremath{\boldsymbol{0}}}}
\newcommand{\boldone}{{\ensuremath{\boldsymbol{1}}}}

\newcommand{\NISOR}{Not-In-Span-Of-Rest }

\DeclareMathOperator{\Ima}{Im}

\DeclareMathOperator*{\Mirror}{\mathbf{\mathscr{M}}}
\DeclareMathOperator*{\CosSet}{\mathbf{\mathscr{C}}}
\newcommand{\ChebMat}{{\ensuremath{\boldsymbol{C}}}} %
\newcommand{\FourMat}{{\ensuremath{\boldsymbol{F}}}} %
\newcommand{\ConMat}{{\ensuremath{\boldsymbol{B}}}} %
\newcommand{\AMat}{{\ensuremath{\boldsymbol{A}}}} %
\DeclareMathOperator*{\abs}{Abs}
\DeclareMathOperator*{\argmax}{arg\,max}

\DeclareMathOperator*{\SX}{\mathcal{X}}
\DeclareMathOperator*{\ST}{\mathcal{T}}
\newcommand{\ChebCoeff}{\hat{\mathsf{c}}}

\newtheorem{theorem}{Theorem}[section]
\newtheorem{lemma}[theorem]{Lemma}
\newtheorem{remark}[theorem]{Remark}
\newtheorem{example}[theorem]{Example}
\newtheorem{corollary}[theorem]{Corollary}

\makeatletter
\def\imod#1{\allowbreak\mkern10mu({\operator@font mod}\,\,#1)}
\makeatother

\numberwithin{equation}{section}
\numberwithin{table}{section}
\numberwithin{figure}{section}

\long\def\symbolfootnote[#1]#2{\begingroup%
\def\thefootnote{\fnsymbol{footnote}}\footnote[#1]{#2}\endgroup}
\newcommand{\sspan}{\textnormal{span}}

\newcommand{\OO}[1]{\mathcal{O}\left(#1\right)}

\renewcommand{\mathbf}[1]{\ensuremath{\boldsymbol{#1}}}
\renewcommand{\textbf}[1]{{\ensuremath{\boldsymbol{#1}}}}
\renewcommand{\thefootnote}{\fnsymbol{footnote}}
\pgfplotscreateplotcyclelist{MR1LOF}{%
dashed, mark=diamond, mark options={solid, scale=1.5}\\%
densely dotted, mark=triangle, mark options={solid, scale=1.5}\\%
dashdotted,mark=x, mark options={solid, scale=1.5}\\%
dashdotted, mark=+, mark options={solid, scale=1.5}\\%
loosely dashed, mark=square, mark options={solid,scale=0.75}\\%
loosely dashdotted, mark=asterisk, mark options={solid,scale=0.5}, opacity=0.5\\%
dashdotdotted, every mark/.append style={solid},mark=star\\%
densely dashdotted,every mark/.append style={solid, fill=gray},mark=diamond*\\%
}%

\title{Constructing efficient spatial discretizations of 
spans of multivariate Chebyshev polynomials}

\date{June 5, 2024}
\author{
Lutz K\"ammerer\footnotemark[1]}

\hypersetup{pdfauthor=Lutz K\"ammerer,
          pdftitle={Efficient discretization of multivariate Chebyshev polynomials},
          pdfsubject={},
          pdfcreator = {pdflatex},
          plainpages=false,
          pdfstartview=FitH,       
          pdfview=FitH,            
          pdfpagemode=UseOutlines, 
          bookmarksnumbered=true, 
          bookmarksopen=false,     
          bookmarksopenlevel=0,   
          colorlinks=true,       
          linkcolor=black,
          citecolor=black,
          urlcolor=black}

\newif\ifshowextendedpaperversion
\showextendedpaperversionfalse

\begin{document}

\maketitle

\begin{abstract}
\small
For an arbitrary given span of high-dimensional multivariate Chebyshev polynomials, an approach to construct spatial discretizations is presented, i.e., the construction of a sampling set that allows for the unique reconstruction of each polynomial of this span.

The approach presented here combines three different types of efficiency.
First, the construction of the spatial discretization should be efficient with respect to the dimension of the span of the Chebyshev polynomials. Second, the number of sampling nodes within the constructed discretizations should be efficient, i.e., the oversampling factors should be reasonable. Third, there should be an efficient method for the unique reconstruction of a polynomial from given sampling values at the sampling nodes of the discretization.

The first two mentioned types of efficiency are also present in constructions based on random sampling nodes, but the lack of structure here causes the inefficiency of the reconstruction method. Our approach uses a combination of cosine transformed rank\mbox{-}1 lattices whose structure allows for applications of univariate fast Fourier transforms for the reconstruction algorithm and is thus a priori efficiently realizable.

Besides the theoretical estimates of numbers of sampling nodes and failure probabilities due to a random draw of the used lattices, we present several improvements of the basic design approach that significantly increases its practical applicability. Numerical tests, which discretize spans of multivariate Chebyshev polynomials depending on up to more than 50 spatial variables, corroborate the theoretical results and the significance of the improvements.
\medskip


\medskip

{\small%
}
\end{abstract}
\footnotetext[1]{
  Chemnitz University of Technology, Faculty of Mathematics, 09107 Chemnitz, Germany\\
  kaemmerer@mathematik.tu-chemnitz.de
}

\medskip

\ifshowextendedpaperversion
\tableofcontents
\newpage
\fi

\section{Introduction}

Polynomial approximation is an important approach in numerical analysis as well as in scientific computing.
In particular, Chebyshev polynomials provide a popular basis of polynomials in univariate as well as multivariate settings \cite{Po67, Tref13,ToTr13,ChDeTrWe18,RaSchw16,PoVo15,PoVo17}.

In lower spatial dimensions $d\in\N$, a standard approach for approximating non-periodic signals using polynomials, cf.\ e.g.\ \cite{GlMa19}, is fixing a maximal degree $N$, sampling the (dilated and shifted) signal at a Cartesian product of one-dimensional grids, i.e., at $(N+1)^d$ sampling nodes in $[-1,1]^d$, and computing the coefficients $(\ChebCoeff_\boldk)_{\boldk\in I}$ of the polynomial
$$
P(\boldx)=\sum_{\boldk\in I}\ChebCoeff_\boldk T_\boldk(\boldx)\,,
$$
by solving a linear system of equations. Here, $I=\left\{\boldk\in\N_0^d\colon\|\boldk\|_\infty\le N\right\}$ is an index set, $\boldx=(x_1,\dots,x_d)^\top$ is a vector of $d\in\N$ variables, and $T_\boldk$ a product of univariate Chebyshev polynomials as defined in \eqref{eq:def_multi_cheb_pol}.
Obviously, even for moderate spatial dimensions $d$ the cardinality of the index set $I$ as well as the cardinality of the Cartesian product grid in spatial domain 
quickly becomes unreasonably large and most applications allow a suitable reduction of the index set $I$ without causing significant deterioration of the approximation quality, see, e.g., \cite{BaNoRi00, ChIwVo21, CoDaLe13, RaSchw16, PoVo17}. 
Most commonly, the corresponding coefficient vector $(\ChebCoeff_\boldk)_{\boldk\in I}$ is computed from samples using a least squares approach, i.e., by solving a normal equation of first kind.
The index set reduction usually leads to less structured index sets and sampling at a complete Cartesian product grid becomes unfavorable due to huge oversampling and excessive computational costs.

Once the possibly less structured index set $I$ of finite cardinality $|I|$ is fixed, the question arises how to efficiently sample the signal in order to guarantee a unique solution of the coefficient vector $(\ChebCoeff_\boldk)_{\boldk\in I}$.
We call a set of sampling nodes $\SX\subset[-1,1]^d$ with finite cardinality $|\SX|$ spatial discretization of the linear span of Chebyshev polynomials $\ChebPol(I):=\sspan\{T_\boldk\colon \boldk\in I\}$, iff the aforementioned least squares approach provides a unique solution, which is given iff the corresponding Chebyshev matrix $\ChebMat(\SX,I)$, cf.~\eqref{eq:cheb_matrix}, has full column rank.

Depending on the individual point of view, discretization approaches are efficient in different ways.
We differ construction efficiency, sample efficiency, and reconstruction efficiency. 
First, approaches that construct sampling sets for adaptively determined index sets $I$ require efficient construction procedures, i.e., the construction of the sampling set should be of manageable complexity.
Second, the number of sampling nodes can be considered efficient if the oversampling factor $|\SX|/|I|$ is not too large. Third, the solution of the least squares approach requires an (approximate) inversion of a matrix-vector product, whose efficiency can be more or less favored by the possible structure of the sampling nodes.

Motivated by the goal of developing a highly efficient dimension-incremental approach for the approximation of non-periodic functions by algebraic polynomials, the aim of this work is to construct spatial discretizations for algebraic polynomials that provide all three types of efficiency mentioned above. Similar to the periodic approach, cf.~\cite{PoVo14, KaPoVo17}, adaptively composed index sets will appear in the non-periodic approach. We are therefore interested in an approach that constructs spatial discretizations for polynomial spaces $\ChebPol(I)$ with arbitrary index sets $I\subset\N_0^d$ of finite cardinality.

Existing sampling approaches based on suitable randomly drawn sampling sets provide construction efficiency as well as sample efficiency, cf.~\cite{CoDaLe13} and columns \textsc{Construction} and $\SX$ in Table~\ref{tab:complexities}.
Due to the associated, generally unstructured Chebyshev matrices, the reconstruction using a conjugate gradient (CG) method suffers from the necessary naive matrix-vector products and is therefore less computationally efficient.
The CG method calculates matrix-vector products with Chebyshev matrices $\ChebMat(\SX,I)$ and their transposed. Since the matrix-vector product with the matrix $\ChebMat(\SX,I)$ corresponds to the evaluation of the polynomial at all sampling nodes of the discretization $\SX$, the computational complexities of (efficient algorithms realizing) the matrix-vector products are given in column \textsc{Evaluation} of Table~\ref{tab:complexities}. 
Note that the complexities specified in this column correspond to those of a single CG iteration.

Further, any linear span of a finite number of multivariate Chebyshev polynomials can be discretized in spatial domain using single cosine transformed rank\mbox{-}1 lattices, resulting in Chebyshev matrices whose inversion can be computed using a fast Fourier transform (FFT) method whose computational effort is almost linear to the number of samples, cf.~\cite{PoVo15, KuoMiNoNu19}. This reconstruction is indeed easy to realize and a single efficiently realizable matrix-vector product. Unfortunately, both the construction efficiency as well as the sample efficiency are far from optimal, cf.~\cite{PoVo15,KuoMiNoNu19} and row CR1L in Table~\ref{tab:complexities}.

The idea of cosine transformed rank\mbox{-}1 lattices is based on the well-known link between the trigonometric system and the Chebyshev system. More precisely, a multivariate polynomial $P\in\ChebPol(I)$, $I\subset\N_0^d$ and $|I|<\infty$, can be regarded as a multivariate trigonometric polynomial $p\in\Pi(J):=\sspan\{\e^{2\pi\ii\boldh\cdot}\colon \boldh\in J\}$ with modified index set $J=\Mirror(I)\subset\Z^d$, cf.~\eqref{eq:cheb_trig}, the so-called mirrored index set $\Mirror(I)$, cf.~\eqref{eq:mirror_operator}.
In fact, a spatial discretization $\ST\subset[0,1)^d$, $|\ST|<\infty$, of $\Pi(\Mirror(I))$, which means that the corresponding Fourier matrix $\FourMat(\ST,\Mirror(I))$, cf.~\eqref{eq:def_FourMat}, is of full column rank, provides a spatial discretization of $\ChebPol(I)$ after cosine transform, cf.\ Theorem~\ref{thm:spatial_discr_per}.
It is also known that sets of rank\mbox{-}1 lattices, suitably drawn at random, provides advantageous spatial discretizations of $\Pi(\Mirror(I))$, cf.~\cite{Kae17}, and thus, can be used as spatial discretizations of $\ChebPol(I)$ after cosine transform. However, a spatial discretization of $\ChebPol(I)$ must fulfill much weaker conditions than (cosine transformed) spatial discretizations of $\Pi(\Mirror(I))$, cf.\ also~\cite{KuoMiNoNu19}.
Accordingly, there is room for improvement which we explore in this paper.
In a first theoretical result, cf.\ Theorem~\ref{thm:prob_bound_T_cheb}, we observe that we can construct spatial discretizations of $\ChebPol(I)$ similar to the construction of spatial discretizations of $\Pi(\Mirror(I))$, cf.~\cite{Kae17}, using a random draw of several rank\mbox{-}1 lattices (and subsequent cosine transform). The significant difference is in the number of drawn rank\mbox{-}1 lattices, which can be significantly reduced from $\lceil C(\ln|\Mirror(I)|-\ln\delta)\rceil$ to $\lceil C(\ln|I|-\ln\delta)\rceil$, where $\delta\in(0,1]$ is a specific upper bound on the failure probability and $C>1/2$ is a universal constant for which $C=2$ seems to be a good choice both theoretically%
 and practically as our numerical tests promise.

The main advantages of the theoretical results in Theorem~\ref{thm:prob_bound_T_cheb} are the observed sample complexity and the obviously simple and highly efficient construction, cf.\ Table~\ref{tab:complexities}. The sample complexity is almost linear in the number $|\Mirror(I)|$. Furthermore, an application of multiple univariate FFTs and a conjugate gradient method provides an efficient reconstruction.
A closer and more practical look at the simple construction approach leads to several improvements, which are presented in three steps in Section~\ref{sec:constr_w_index_set}.
In the course of this, we make sure that the sample complexity is essentially preserved - apart from exceptional cases that are less relevant in practice - and that the upper bound on the failure probability is also preserved.
Certainly, in each step the computational effort of the construction algorithm grows.
In the first improvement step, the construction complexity increases to $\OO{|\Mirror(I)|(d+\log|I|)\log|I|}$.
In each of the subsequent two steps, the computational costs for construction increase by factors in $\OO{(d+\log|I|)\log|I|}$.
Please note that $|\Mirror(I)|\le 2^d|I|$ applies in general and is a worst case estimate. Some estimates on the given sample complexities in column $|\SX|$ of Table~\ref{tab:complexities} 
suffer from this estimate. As a result, the determined evaluation complexities are also affected.
To simplify the comparison, the two terms $|\Mirror(I)|$ and $2^d|I|$ can be used synonymously.

\begin{table}\small
\begin{tabular}{lccc}
\toprule
&\multicolumn{3}{c}{\textsc{Complexities}}\\
\midrule
\textsc{Approach} & $|\SX|$ & \textsc{Construction} & \textsc{Evaluation}\\
\midrule
random nodes \cite{CoDaLe13} & $|\Mirror(I)|\log|\Mirror(I)|$
& $d\,|\Mirror(I)|\log|\Mirror(I)|$
& $|\Mirror(I)|\,|I|\log|\Mirror(I)|$
\\
\rule{0px}{14pt}
CR1L \cite{KuoMiNoNu19}
& $|\Mirror(I)|\,|I|$
& $d\,|\Mirror(I)|\,|I|$
& $|\Mirror(I)|\,|I|\,(\log|\Mirror(I)|)^2$
\\
\midrule
Theorem~\ref{thm:prob_bound_T_cheb} & $|\Mirror(I)| \log|I|$ & $d\,|I|$ & $|\Mirror(I)|\,(d+\log|I|)\log|I|$\\
\rule{0px}{14pt}Section~\ref{sssec:greedy_improvement} & $|\Mirror(I)| \log|I|$ & $|\Mirror(I)|\,(d+\log|I|)\log|I|$ & $|\Mirror(I)|\,(d+\log|I|)\log|I|$\\
\rule{0px}{14pt}Section~\ref{sssec:improvement_simple} & $2^d\,|I|\log|I|$ & $|\Mirror(I)|\,(d+\log|I|)(\log|I|)^3$ & $d\, 2^d\, |I|\,(\log|I|)^2$\\
\rule{0px}{14pt}Section~\ref{sssec:further_improvement} & $2^d\,|I|\log|I|$ & $|\Mirror(I)|\,(d+\log|I|)^2(\log|I|)^3$ & $d\, 2^d\, |I|\,(\log|I|)^2$\\
\bottomrule
\end{tabular}
\caption{\small Complexities of cardinalities $|\SX|$ of spatial discretizations $\SX$ for $\ChebPol(I)$ under the assumption $N_I\lesssim |\Mirror(I)|/\log|I|$, cf.~\eqref{eq:max_occur_powers}, 
construction complexities, and matrix evaluation complexities, i.e., computational complexities of efficient realizations of matrix vector products with matrices $\ChebMat(\SX,I)$ and $\ChebMat(\SX,I)^\top$.
}\label{tab:complexities}
\end{table}

\begin{remark}\label{rem:ds_dependence}
Note that the dependencies on the spatial dimension $d$ of the listed results in Table~\ref{tab:complexities} are based on estimates without additional assumptions on, e.g., the simultaneously active dimensions.
A common assumption, because it is observed in many applications, is that significant individual basis functions of approximations do not actually depend on all $d$ possible variables. In fact an analysis of variance (ANOVA) can reveal significant variables and their interactions, cf.~\cite{RaAl99}. Already available numerical methods for decomposing high-dimensional signals w.r.t.~their variance ranking of the variables and present interactions, cf.~\cite{PoSchmi19, PoSc19b}, in a black-box scenario can benefit from the discretization strategies developed in this paper. However, the estimates on the number of sampling nodes should be refined for this purpose.

If one assumes that each basis function $T_\boldk$ of $\ChebPol(I)$ actually depends on at most $d_s<d$ different variables $x_j$, $j\in\{1,\dots,d\}$, the listed estimates in Table~\ref{tab:complexities} can be reduced significantly. To be more precise, the given dependencies on the dimension $d$ can be mitigated to similar dependencies on $d_s$, since then cardinalities $|\Mirror(I)|$ of the mirrored index set can be bounded from above by $|\Mirror(I)|\le 2^{d_s}|I|$. Moreover, if the index set $I\subset\N_0^d$ fulfills a so-called downward closed or lower set property, cf.~\cite{ChCoMiNoTe15}, it is possible to estimate the cardinality  $|\Mirror(I)|\le |I|^{\ln 3/\ln 2}$, which reduces the given bounds in certain scenarios.

In addition, we need to mention that the component-by-component construction of cosine transformed single rank\mbox{-}1 lattices in \cite{KuoMiNoNu19} is a deterministic approach, i.e., it is not subject to any probability of failure. All other constructions mentioned can be affected by failures that occur with well controllable and low probabilities. Moreover, for given cosine transformed multiple rank\mbox{-}1 lattices determined by Theorem~\ref{thm:prob_bound_T_cheb}, it is relatively easy and computationally affordable to check a sufficient discretization condition. The improved approaches from Section~\ref{sec:constr_w_index_set} are based on this check so that possible failures can be reported at runtime.
\end{remark}

The paper is organized as follows. In Section~\ref{sec:pre} basic facts from linear algebra, the connection of the Chebyshev and the trigonometric system, and rank\mbox{-}1 lattices are collected. Section~\ref{sec:JoiningSamplingSets} presents the general approach for constructing spatial discretizations step by step using what we call \NISOR condition here. A general improvement idea is discussed and, in addition, how the connection between the Chebyshev system and the trigonometric system can be exploited to construct the discretization for $\ChebPol(I)$ in the trigonometric framework.
In Section~\ref{sec:theory} we adapt some results for sampling along multiple rank\mbox{-}1 lattices in the trigonometric system in order to obtain specific basic results for the Chebyshev system. Section~\ref{sec:improvements} presents certain improvements that prove particularly useful in several numerical tests. Some of these tests can be found in Section~\ref{sec:numtestcheb}.

\subsubsection*{Notation}
We use bold Latin uppercase letters as symbols for matrices and bold Latin lowercase letters for vectors and denote the kernel of a $n\times m-$matrix $\AMat$ by $\ker(\AMat)$ and its image by $\Ima(\AMat)$. Moreover, the inner product $\bolda \cdot \boldb$ of two vectors $\bolda=(a_1,\dots,a_d)^\top\in\C^d$ and $\boldb=(b_1,\dots,b_d)^\top\in\C^d$ is defined by $\bolda\cdot\boldb:= \sum_{j=1}^da_j\overline{b_j}$.

We denote by $\delta_{0}\colon \C\to \{0,1\}$, $a\mapsto\begin{cases}1,& a=0,\\ 0, &\text{otherwise},\end{cases}$ the Kronecker delta function and we use
$\|\cdot\|_0\colon\C^d\to\N_0$, $\bolda\mapsto d-\sum_{j=1}^d\delta_{0}(|a_j|)$ for counting the number of nonzero elements within complex, real, or integer valued vectors. For $p\in[1,\infty]$ the mapping $\|\cdot\|_p\colon\C^d\to\R$ is the usual $p$-norm for vectors in $\C^d$.
Furthermore, for $I\subset \N_0^d$ we define the mirror operator
\begin{equation}
\Mirror(I):=\{\boldh=(l_1k_1,\dots,l_dk_d)^\top\colon \boldk\in I, \boldl\in\{-1,1\}^d\}
\label{eq:mirror_operator}
\end{equation}
and for sets $\ST\subset\T^d$, where $\T:=\R/\Z$ is the torus, we define the set of cosine transformed elements of the set $\ST$ by
$$
\CosSet(\ST):=\{\boldx=(\cos(2\pi t_1),\dots, \cos(2\pi t_d))^\top\in[-1,1]^d\colon\boldt\in\ST\}\,.
$$
In addition, we denote the cardinality of a set $I$ by $|I|$ and for $I\subset\R^d$ the symbol
$\abs(I)$ denotes the set
$$
\abs(I):=\{\boldk=|\boldh|\colon \boldh\in I\}\,,
$$
where $|\boldh|:=(|h_1|,\dots, |h_d|)^\top$ is the vector containing the absolute values of the components of $\boldh$. 
Moreover, we define a function computing the next prime number greater than its input by
$$
\operatorname{nextprime}\colon\R\to \N,\qquad x\mapsto\min\{p \colon p>x,\, p\text{ prime}\}\,.
$$
For sets $J$ of finite cardinality $|J|$, we will consider vectors $\bolda=(a_j)_{j\in J}=(a_{j_1},\dots,a_{j_{|J|}})^\top\in\C^{|J|}$, where we assume that the distinct elements $j_1,\dots,j_{|J|}$ of $J$ are in a fixed order. We also assume (the same) fixed orders of sets when indexing matrix elements.

\section{Linear algebra, Fourier vs.\ Chebyshev, rank\mbox{-}1 Lattices}\label{sec:pre}

In this section, we collect basic facts about the ingredients that we combine to obtain the main result of this work. Each of the ingredients is considered more or less separately in this section.

\subsection{Linear algebra}\label{ssec:linalg}

We consider matrices $\AMat\in\C^{n\times m}$ and we denote the $l$th column of the matrix $\AMat$ by $\bolda_l$, $l\in\{1,\dots, m\}$.
If $\bolda_l$ fulfills 
$$\bolda_l\not\in\sspan\left\{\bolda_j\colon j\in\{1,\dots,m\}\setminus\{l\}\right\}\,,$$
we say that
$\bolda_l$ fulfills the \emph{\NISOR condition w.r.t.\ $\AMat$}. This condition will be widely used throughout the work.
In running text we will use the expression \emph{\NISOR condition}, otherwise the formal description.
In cases where the connection to matrix $\AMat$ is clear, we omit "w.r.t.\ $\AMat$".
We compile some basic facts that we can refer to later.

Without loss of generality, we just consider the first columns of matrices due to notation simplification. Clearly, all statements apply analogously to any fixed column of the matrices.

\begin{lemma}\label{lem:span_subset}
Let the matrix $\AMat=\left(a_{kj}\right)_{k=1,\dots,n; j=1,\dots,m}\in\C^{n\times m}$ be given.
For a subset $K\subset\{1,\dots,n\}$ we denote 
$$
\tilde{\bolda}_j=\left(a_{kj}\right)_{k\in K},\; j=1,\dots,m,
$$
where the elements in $K$ are assumed to be naturally ordered.

If there exists a set $K\subset\{1,\dots,n\}$ such that
$$\tilde{\bolda}_1\not\in\sspan\left\{\tilde{\bolda}_j\colon j\in\{2,\dots,m\}\right\}\,,$$
then $\bolda_1\not\in\sspan\left\{\bolda_j\colon j\in\{2,\dots,m\}\right\}$ holds.
\end{lemma}
\begin{proof}
Clearly, if the linear system of equations
$
\tilde{\bolda}_1=\sum_{j\in\{2,\dots,m\}}\alpha_j \tilde{\bolda}_{j}
$
does not allow for an exact solution, then the even more overdetermined linear system of equations
$
\bolda_1=\sum_{j\in\{2,\dots,m\}}\alpha_j \bolda_{j}
$
cannot yield an exact solution.
\end{proof}

Later, we will deal with sampling matrices, i.e., the rows of the matrices will be determined by sampling nodes. Since our construction of sets of sampling nodes is affected by 
duplicates, we need to take into account the disappearance of duplicate rows of matrices. 
To this end, we say that matrices $\AMat^{(1)}\in\C^{n_1\times m}$ and $\AMat^{(2)}\in\C^{n_2\times m}$, $m,n_1,n_2\in\N$, are essentially the same and we denote $\AMat^{(1)}\triangleq\AMat^{(2)}$ iff the sets of the row vectors of both matrices coincide, i.e., by eliminating duplicate rows in both matrices and applying a suitable permutation to the remaining rows of one of the matrices equality of the matrices can be achieved.

\begin{lemma}\label{lem:esential_step_1}
Let $\AMat^{(1)}\in\C^{n\times m}$ be given. We define $\AMat^{(2)}\in\C^{k\times m}$ as the matrix that consists of all $k$ pairwise distinct rows of $\AMat^{(1)}$. In addition, we denote $\bolda^{(i)}_l$ as the $l$th column of $\AMat^{(i)}$, $i=1,2$. Then
$$
\bolda_1^{(1)}\not\in \sspan\{\bolda_l^{(1)}\colon l\in\{2,\dots,m\}\}
\Leftrightarrow
\bolda_1^{(2)}\not\in \sspan\{\bolda_l^{(2)}\colon l\in\{2,\dots,m\}\}
$$
holds.
\end{lemma}
\begin{proof}
Dropping the duplicate rows of the linear system of equations
$$
\bolda_1^{(1)}=\sum_{k=2}^n\alpha_k\bolda_k^{(1)}
$$
does not affect the existence of a solution as well as a possible solution and the remaining linear system of equations coincides with
$\bolda_1^{(2)}=\sum_{k=2}^n\alpha_k\bolda_k^{(2)}$.
\end{proof}

\begin{lemma}\label{lem:essentially_span}
Let $\AMat^{(1)}\in\C^{n_1\times m}$ and $\AMat^{(2)}\in\C^{n_2\times m}$ be given and we assume $\AMat^{(1)}\triangleq\AMat^{(2)}$, i.e.,
the matrices $\AMat^{(1)}$ and $\AMat^{(2)}$ are essentially the same.
Then
$$
\bolda_1^{(1)}\not\in \sspan\{\bolda_l^{(1)}\colon l\in\{2,\dots,m\}\}
\Leftrightarrow
\bolda_1^{(2)}\not\in \sspan\{\bolda_l^{(2)}\colon l\in\{2,\dots,m\}\}\,.
$$
holds.
\end{lemma}
\begin{proof}
Applying Lemma~\ref{lem:esential_step_1} twice yields the assertion.
\end{proof}

\begin{lemma}\label{lem:matrices} 
Let the matrix $\AMat\in\C^{n\times m}$ of the following form
$$
\AMat=\begin{pmatrix}
\AMat_{11}& \AMat_{12} \\ \AMat_{21} & \AMat_{22} 
\end{pmatrix}
$$
with submatrices $\AMat_{ij}\in\C^{n_i \times m_j}$, $n_i,m_j\in\N_0$, $i\in\{1,2\}$, $j\in\{1,2\}$, $n_1+n_2=n$, $m_1+m_2=m$, be given.
In addition, we assume that
each column of $\AMat_{11}$ fulfills the \NISOR condition with respect to $\begin{pmatrix}
\AMat_{11}& \AMat_{12} 
\end{pmatrix}$.

If the $l$th column of $\AMat_{22}$, $1\le l\le m_2$,  fulfills the \NISOR condition w.r.t.\ $\AMat_{22}$,
then the $(m_1+l)$th column of the matrix $\AMat$ fulfills the \NISOR condition w.r.t.\ the matrix $\AMat$.
\end{lemma}

\begin{proof}
We assume the contrary, i.e. under the assumptions of the lemma, we assume that the $(m_1+l)$th column of $\AMat$ does not fulfill 
the \NISOR condition w.r.t.\ the matrix $\AMat$. Accordingly, there exist $\alpha_k\in\C$, $k\in\{1,\dots,m\}\setminus\{m_1+l\}$, such that
$$
\bolda_{m_1+l}=\sum_{\substack{k=1\\k\neq m_1+l}}^m\alpha_k\bolda_k
$$
holds.
Due to the fact that each column of $\AMat_{11}$ fulfills the \NISOR condition w.r.t.\ $\begin{pmatrix}
\AMat_{11}& \AMat_{12} 
\end{pmatrix}$, we observe $\alpha_k=0$ for $k\in\{1,\dots,m_1\}$.
Accordingly, we observe that $\bolda_{m_1+l}\in\sspan\{\bolda_j\colon j\in\{m_1+1,\dots,m\}\setminus\{m_1+l\}\}$ and, in particular, that the $l$th
column of $\AMat_{22}$ is in the span of the other columns of $\AMat_{22}$ 
which is in contradiction to the assumptions.
\end{proof}

\begin{remark}\label{rem:NISOR_at_least_two_columns}
Due to theoretical considerations, we distinguish three types of matrices $\AMat\in\C^{n\times m}$. On the one hand, there are matrices where each column fulfills the \NISOR condition, which means that the matrix has full column rank.
On the other hand, there are matrices where at least one column $\bolda_j$, $j\in\{1,\dots,m\}$, does not fulfill the \NISOR condition, which leads to a distinction of the two other types. First, assuming that there is only a single column which does not fulfill the \NISOR condition, we observe $\bolda_j=\boldzero$. Second, if $\bolda_j\neq\boldzero$, then at least two columns do not fulfill the \NISOR condition. Accordingly, it is not possible that only a single nonzero column does not fulfill the \NISOR condition.
\end{remark}

\subsection{Fourier vs.\ Chebyshev}

For $k\in\N_0$, we define the univariate Chebyshev polynomials
\begin{align*}
T_k\colon [-1,1]\to\R, \qquad x\mapsto 2^{\frac{1-\delta_0(k)}{2}}\cos(k\arccos(x))\,,%
\end{align*}
i.e.,
\begin{align*}
T_k(x)=\begin{cases}1&\text{ $k=0$,}\\\sqrt{2}\cos(k\arccos(x)) & \text{ $k\in\N$,}\end{cases}
\end{align*}
which constitute an orthonormal system in $L^2([-1,1],\rho_1)$, i.e., w.r.t.\ the inner product
$$
\langle\cdot,\cdot\rangle_1\colon L^2([-1,1],\rho_1)\times L^2([-1,1],\rho_1)\to\R,\qquad (f,g)\mapsto\langle f,g\rangle_1:=\int_{-1}^1f(x)g(x)\rho_1(x)\mathrm{d}x
$$
where $\rho_1(x):=(\pi\sqrt{1-x^2})^{-1}$ is the weight function.
For $d\in\N$ and $\boldk\in\N_0^d$, we define the $d$-variate Chebyshev polynomials
\begin{align}
T_\boldk\colon[-1,1]^d\to\R,\qquad \boldx=(x_1,\dots,x_d)^\top\mapsto \prod_{j=1}^dT_{k_j}(x_j)\,,\label{eq:def_multi_cheb_pol}
\end{align}
which constitute an orthonormal system in  $L^2([-1,1]^d,\rho_d)$, i.e., w.r.t.\ the inner product
$$
\langle\cdot,\cdot\rangle_d\colon L^2([-1,1]^d,\rho_d)\times L^2([-1,1]^d,\rho_d)\to\R,\quad (f,g)\mapsto \langle f,g\rangle_d:=\int_{[-1,1]^d}f(\boldx)g(\boldx)\rho_d(\boldx)\mathrm{d}\boldx\,,
$$
where $\rho_d(\boldx)=\prod_{j=1}^d\rho_1(x_j)$.

We observe the equality
\begin{align}
T_\boldk(\boldx)&=\prod_{j=1}^dT_{k_j}(x_j)=\prod_{j=1}^d2^{\frac{1-\delta_0(k_j)}{2}}\cos(k_j\arccos(x_j))=\prod_{\substack{j=1\\k_j\neq 0}}^{d}\frac{\e^{\ii k_j\arccos(x_j)}+\e^{-\ii k_j\arccos(x_j)}}{\sqrt{2}}\nonumber\\
&=2^{-\frac{\|\boldk\|_0}{2}}\prod_{\substack{j=1\\k_j\neq 0}}^d\left(\e^{\ii k_j\arccos(x_j)}+\e^{-\ii k_j\arccos(x_j)}\right)\nonumber\\
&=2^{-\frac{\|\boldk\|_0}{2}}\sum_{\boldh\in\Mirror(\{\boldk\})}\prod_{\substack{j=1\\h_j\neq 0}}^d\e^{\ii h_j\arccos(x_j)}=
\sum_{\boldh\in\Mirror(\{\boldk\})}2^{-\frac{\|\boldh\|_0}{2}}\e^{\ii \boldh\cdot\arccos(\boldx)},
\label{eq:four_cheb_scaling}
\end{align}
where the $\arccos$ is applied to vectors componentwise.

For a Chebyshev coefficient series $\{\ChebCoeff_\boldk\}_{\boldk\in\N_0^d}\in \ell_1(\N_0^d)$, we observe
\begin{align*}
\sum_{\boldk\in\N_0^d}\ChebCoeff_\boldk T_\boldk(\boldx)
=\sum_{\boldh\in\Z^d}2^{-\frac{\|\boldh\|_0}{2}}\,\ChebCoeff_{|\boldh|}\e^{\ii \boldh\cdot\arccos(\boldx)},
\end{align*}
where $|\boldh|$ is the vector of componentwise absolute values of $\boldh$.
Moreover, for an arbitrary $d$-variate polynomial $P\in\ChebPol(I)$, $I\subset\N_0^d$, $|I|<\infty$, i.e., $P(\boldx)=\sum_{\boldk\in I}\ChebCoeff_\boldk T_\boldk(\boldx)$ with $(\ChebCoeff_\boldk)_{\boldk\in I}\in\R^{|I|}$,
we observe
\begin{align}
P(\boldx)&=\sum_{\boldk\in I}2^{-\frac{\|\boldk\|_0}{2}}\,\ChebCoeff_{\boldk}\sum_{\boldh\in\Mirror(\{\boldk\})}\e^{\ii \boldh\cdot\arccos(\boldx)}
=\sum_{\boldk\in I}\sum_{\boldh\in\Mirror(\{\boldk\})}2^{-\frac{\|\boldh\|_0}{2}}\,\ChebCoeff_{|\boldh|}\e^{\ii \boldh\cdot\arccos(\boldx)}\nonumber\\
&=
\sum_{\boldh\in\Mirror(I)}\underbrace{2^{-\frac{\|\boldh\|_0}{2}}\,\ChebCoeff_{|\boldh|}}_{=:\hat{p}_\boldh}\e^{2\pi\ii \boldh\cdot\frac{\arccos(\boldx)}{2\pi}}
=:p\left(\frac{\arccos(\boldx)}{2\pi}\right),\label{eq:cheb_trig}
\end{align}
where $p\colon\T^d\to\C$ is a multivariate trigonometric polynomial. For each $\boldx\in[-1,1]^d$, we have $\boldt=\frac{\arccos(\boldx)}{2\pi}\in[0,1/2]^d$, and consequently, $p(\boldt)=P(\cos(2\pi \boldt))=P(\boldx)$ holds. The similar equality also holds for $\boldt\in[0,1)^d\setminus[0,1/2]^d$, due to the fact that for
\begin{align*}
\tilde{t}_j:=\begin{cases}
t_j, & t_j\in[0,1/2],\\
1-t_j , & t_j\in(1/2,1),
\end{cases}
\end{align*}
we observe $\tilde{\boldt}\in[0,1/2]^d$ and
\begin{align*}
p(\tilde{\boldt})=\sum_{\boldk\in I}2^{\frac{-\|\boldk\|_0}{2}} \ChebCoeff_{\boldk}\sum_{\boldh\in\Mirror(\{\boldk\})}\e^{2\pi\ii \boldh\cdot\tilde{\boldt}}
=\sum_{\boldk\in I}2^{\frac{-\|\boldk\|_0}{2}} \ChebCoeff_{\boldk}\sum_{\boldh\in\Mirror(\{\boldk\})}\e^{2\pi\ii \boldh\cdot{\boldt}}=p(\boldt)
\end{align*}
due to the symmetry of each $\Mirror(\{\boldk\})$.

Moreover, we observe
$$
\boldx=\cos(2\pi\tilde\boldt)=\cos(2\pi\boldt),
$$
due to $\cos(2\pi \tilde{t}_j)=\cos(-2\pi \tilde{t}_j)=\cos(2\pi (1-\tilde{t}_j))=\cos(2\pi t_j)$.
Accordingly, we ascertain $P(\boldx)=P(\cos(2\pi\boldt))=p(\boldt)$ for each $\boldt\in[0,1)^d$ and thus for
all $\boldt\in\T^d$.

For $\mathcal{X}\subset[-1,1]^d$ and $I\subset\N_0^d$, we denote the Chebyshev evaluation matrix by
\begin{equation}
\ChebMat(\SX,I):=
\left(T_\boldk(\boldx)\right)_{\boldx\in\SX,\boldk\in I}\in\R^{|\SX|\times |I|},
\label{eq:cheb_matrix}
\end{equation}
and observe that the evaluation of $P\in\ChebPol(I)$ at all $x\in\SX$ is equivalent to the matrix vector product
$
\ChebMat(\SX,I)\,{\mathbf \ChebCoeff}=\boldP
$, where ${\mathbf \ChebCoeff}=\left(\ChebCoeff_\boldk\right)_{\boldk\in I}$ is the vector of Chebyshev coefficients corresponding to the algebraic polynomial $P$ and $\boldP=(P(\boldx))_{\boldx\in\mathcal{X}}$. Clearly, the unique reconstruction of the polynomial $P$ from sampling values $P(\boldx)$, $\boldx\in\SX$, i.e., the unique reconstruction of $\boldc$ from these sampling values, is guaranteed
if and only if the matrix $\ChebMat(\SX,I)$ has full column rank $|I|$.

Due to the last considerations, we can decompose the matrix $\ChebMat(\mathcal{X},I)$ in two matrices $\FourMat$ and $\ConMat$, where the matrix $\ConMat$ computes the connection between $\ChebCoeff_\boldk$ and $\hat{p}_\boldh$ for all $\boldh\in\Mirror(\{\boldk\})$ and all $\boldk\in I$, cf.~\eqref{eq:cheb_trig}, and $\FourMat$ is a specific Fourier evaluation matrix, which is generally defined by
\begin{equation}
\FourMat(\ST,J)=\left(\e^{2\pi\ii\boldh\cdot\boldt}\right)_{\boldt\in\ST,\boldh\in J}\in\C^{|\ST|\times|J|}\,,
\label{eq:def_FourMat}
\end{equation}
where $\ST\subset\T^d$ is a sampling set and $J\subset\Z^d$ an index set that specifies a set of trigonometric monomials. The matrix $\ConMat$ is given by
$$
\ConMat(I):=\left(b_{\boldh,\boldk}\right)_{\boldh\in\Mirror(I),\boldk\in I}\in\R^{|\Mirror(I)|\times|I|}\,,
$$
where
$$
b_{\boldh,\boldk}:=
\begin{cases}
2^{-\frac{\|\boldk\|_0}{2}},& \boldh\in\Mirror(\{\boldk\}),\\
0,& \text{otherwise}.
\end{cases}
$$
Clearly, for given $\SX\subset[-1,1]^d$, $|\SX|<\infty$, and $\ST:=\{\boldt=\frac{\arccos(\boldx)}{2\pi}\colon \boldx\in\SX\}$
the relation
$$
\ChebMat(\SX,I)\triangleq\FourMat(\ST,\Mirror(I))\ConMat(I)
$$
holds, i.e., $\ChebMat(\SX,I)$ and $\FourMat(\ST,\Mirror(I))\ConMat(I)$ are essentially the same matrices. Due to the considerations above, we also have the same relation for given $\ST\subset\T^d$, $|\ST|<\infty$, and determined $\SX=\CosSet(\ST):=\{\cos(2\pi\boldt)\colon\boldt\in\mathcal{T}\}$.
At this point, we would like to point out that the sets $\ST$ and $\SX$ might not have the same cardinality due to several possible elements $\boldt_1,\boldt_2\in\ST$, $\boldt_1\neq\boldt_2$, for which $\cos(2\pi\boldt_1)=\cos(2\pi\boldt_2)$ holds. However, the sets of the row vectors of the matrices $\ChebMat(\SX,I)$ and $\FourMat(\ST,\Mirror(I))\ConMat(I)$ coincide, which means that the matrices are essentially the same matrices, cf.\ Section~\ref{ssec:linalg}.

The matrix $\ConMat(I)$ is highly sparse and structured. More precisely, each row of $\ConMat(I)$ has exactly one nonzero entry, since each $\boldh\in\Mirror(I)$ is an element of exactly one of the sets $\Mirror(\{\boldk\})$, $\boldk\in I$. As a consequence, $\ConMat(I)$ has full column rank $|I|$ by default. We observe the following facts.
\begin{lemma}\label{lem:C_full_crank}
Let $ \SX\subset[-1,1]^d$ or $\ST\subset\T^d$ as well as $I\subset\N_0^d$ be given. 
Moreover, we assume that $\SX=\CosSet(\ST)$ holds.
Then the matrix $\ChebMat(\SX,I)$ has full column rank $|I|$ if and only if
$\ker(\FourMat(\ST,\Mirror(I)))\cap \Ima(\ConMat(I))=\{\boldzero\}$.
\end{lemma}
\begin{proof}
This is a simple consequence of the Rank-nullity theorem.
\end{proof}
The last lemma, states, that each sampling set $\ST\subset\T^d$ which provides
$\ker(\FourMat(\ST,\Mirror(I)))\cap \Ima(\ConMat(I))=\{\boldzero\}$ can be cosine transformed to a sampling set $\SX$, which is then a spatial discretization of $\ChebPol(I)$. Clearly, each sampling set $\ST$ with $\ker(\FourMat(\ST,\Mirror(I)))=\{\boldzero\}$ provides this property.

\begin{theorem}\label{thm:spatial_discr_per}
Let $\ST$ be a spatial discretization of $\Pi(\Mirror(I))$. Then $\SX=\CosSet(\ST)$ is a spatial discretization of $\ChebPol(I)$.
\end{theorem}
\begin{proof}
Since $\ST$ is a spatial discretization of $\Pi(\Mirror(I))$, we observe $\ker(\FourMat(\ST,\Mirror(I)))=\{\boldzero\}$ and thus $\ChebMat(\SX,I)\triangleq\FourMat(\ST,\Mirror(I))\ConMat(I)$ is of full column rank due to Lemma~\ref{lem:C_full_crank}.
\end{proof}

Certainly, the case $\ker(\FourMat(\ST,\Mirror(I)))=\{\boldzero\}$ is the most strict sufficient condition that implies $\ker(\FourMat(\ST,\Mirror(I))) \cap \Ima(\ConMat(I))=\{\boldzero\}$. Since satisfying less stringent requirements on the matrix $\FourMat(\ST,\Mirror(I))$ promises a lower number $|\ST|$ and, consequently, a possibly lower number $|\CosSet(\ST)|$ of sampling nodes, we are interested in less stringent but still sufficient conditions on the matrix $\FourMat(\ST,\Mirror(I))$ such that $\ChebMat(\CosSet(\ST),I)$ has full column rank.

\begin{lemma}\label{lem:C_linear independent_columns}
Let $\ST
\subset\T^d$, $|\ST|<\infty$,  and $I\subset\N_0^d$, $|I|<\infty$, be given. We consider the matrices
$\FourMat(\ST,\Mirror(I))$ and $\ConMat(I)$
and the corresponding matrices $\ChebMat(\CosSet(\ST),I)\triangleq \FourMat(\ST,\Mirror(I))\ConMat(I)$.
We denote the columns of the matrices $\ChebMat(\CosSet(\ST),I)$ and $\FourMat(\ST,\Mirror(I))$ as $\boldc_{\boldk}$, $\boldk\in I$, and
$\bolda_{\boldh}$, $\boldh\in\Mirror(I)$.

In addition, we assume that $\bolda_{\boldh}\not\in\sspan\{\bolda_{\boldl}\colon \boldl\in\Mirror(I)\setminus\{\boldh\}\}$ holds. Then we observe $\boldc_{|\boldh|}\not\in\sspan\{\boldc_{\boldl}\colon \boldl\in I\setminus\{|\boldh|\}\}$.
\end{lemma}

\begin{proof}
Due to the definition of the matrix $\ConMat(I)$, we observe for $j\in\{1,\dots,|\ST|\}$ and $\boldk\in I$
\begin{align*}
c_{j,\boldk}'&:=\sum_{\boldl\in\Mirror(I)}a_{j,\boldl} b_{\boldl,\boldk}=
\sum_{\boldl\in\Mirror(I)}a_{j,\boldl}2^{-\frac{\|\boldk\|_0}{2}}\delta_0(\|\boldk-|\boldl|\|)\\
&=2^{-\frac{\|\boldk\|_0}{2}}\sum_{\boldl\in\Mirror(\{\boldk\})}a_{j,\boldl}
\end{align*}
which yields that the $\boldk$th column $\boldc_{\boldk}'$ of $\FourMat(\ST,\Mirror(I))\ConMat(I)$ is given by  $\boldc_{\boldk}'= 2^{-\frac{\|\boldk\|_0}{2}}\sum_{\boldl\in\Mirror(\{\boldk\})}\bolda_{\boldl}$.
Due to the fact that each $\bolda_{\boldl}$, $\boldl\in\Mirror(I)$, contributes to exactly one $\boldc_{\boldk}'$, $\boldk\in I$, and each of the $\boldc_{\boldk}'$ is such a linear combination of a specific set of 
$\bolda_{\boldl}$, $\boldl\in\Mirror(\{\boldk\})$ , we observe $\boldc_{|\boldh|}'\not\in\sspan\{\boldc_{\boldl}'\colon \boldl\in I\setminus\{|\boldh|\}\}$. The matrix $\ChebMat(\CosSet(\ST),I)$ is essentially the same as the matrix consisting of the columns $\boldc_{\boldk}'$, $\boldk\in I$, i.e., applying Lemma~\ref{lem:essentially_span}, we observe $\boldc_{|\boldh|}\not\in\sspan\{\boldc_{\boldl}\colon \boldl\in I\setminus\{|\boldh|\}\}$.
\end{proof}

\begin{remark}\label{rem:weakened_assumption}The assumption $\bolda_{\boldh}\not\in\sspan\{\bolda_{\boldl}\colon \boldl\in\Mirror(I)\setminus\{\boldh\}\}$ in Lemma~\ref{lem:C_linear independent_columns}
can be weakened slightly to
$\bolda_{\boldh}\not\in\sspan\{\,\bolda_{\boldl}\colon \boldl\in\Mirror(I\setminus\{|\boldh|\})\,\}$ in combination with the assumption $\boldc_{|\boldh|}'\cdot\bolda_{\boldh}\neq 0$. We obtain the same result, i.e.,
 $\boldc_{|\boldh|}\not\in\sspan\{\boldc_{\boldl}\colon \boldl\in I\setminus\{|\boldh|\}\}$.
\end{remark}

Lemma~\ref{lem:C_linear independent_columns}
provides a sufficient condition on the columns of matrices $\FourMat(\ST,\Mirror(I))$ which implies that $\ChebMat(\CosSet(\ST),I)$ are full column rank matrices.
In particular, it is not necessary that all the columns of $\FourMat(\ST,\Mirror(I))$ are linearly independent, i.e., $\ker(\FourMat(\ST,\Mirror(I)))=\{\boldzero\}$. It is enough that there exists a set $J\subset\Mirror(I)$ such that
\begin{itemize}
\item for each $\boldh\in J$ the matrix column $\bolda_\boldh\not\in\sspan\{\bolda_\boldl\colon\boldl\in\Mirror(I)\setminus\{\boldh\}\}$ and
\item $I\subset 
\abs(J)$, which implies equality of these two sets.
\end{itemize}
Then, each column of $\ChebMat(\CosSet(\ST),I)$ is not in the span of all other columns, i.e., the columns of 
$\ChebMat(\CosSet(\ST),I)$ are linearly independent.

\begin{corollary}\label{cor:matrix_C_full_crank}
Let an index set $I\subset\N_0^d$, $|I|<\infty$, and a sampling set $\ST\subset\T^d$, $|\ST|<\infty$, be given. We denote the columns of $\FourMat(\ST,\Mirror(I))$ by $\bolda_\boldh$, $\boldh\in \Mirror(I)$.
We assume that for each $\boldk\in I$ there exists $\bolda_{\boldh}$, $\boldh\in\Mirror(\{\boldk\})$, such that $\bolda_{\boldh}\not\in\sspan\{\bolda_{\boldl}\colon \boldl\in\Mirror(I)\setminus\{\boldh\}\}$ holds. Then the matrix $\ChebMat(\CosSet(\ST),I)
$ has full column rank.
\end{corollary}
Corollary \ref{cor:matrix_C_full_crank} yields a construction concept for a spatial discretization of spans of Chebyshev polynomials. We just need to construct a sampling set $\mathcal{T}\subset\T^d$ such that for each $\boldk\in I$ at least one of the columns $\bolda_{\boldh}$, $\boldh\in\Mirror(\{\boldk\})$, is not in the span of all other columns $\bolda_{\boldl}$, $\boldl\in\Mirror(I)\setminus\{\boldh\}$.
Then the sampling set $\CosSet(\ST)=\{\cos(2\pi \boldt)\colon \boldt\in\mathcal{T}\}$ is a spatial discretization of $\ChebPol(I)$.

\subsection{rank\mbox{-}1 Lattices}
One possibility for discretizing multivariate trigonometric polynomials are so-called rank\mbox{-}1 lattices, cf.~\cite{Kae2013, KuoMiNoNu19}. For a given \emph{generating vector} $\boldz\in\Z^d$ and \emph{lattice size} $M\in\N$, we call the set
$$\Lambda(\boldz,M):=\left\{\frac{j}{M}\boldz\bmod \boldone\colon j\in\{0,\dots,M-1\}\right\}\subset\T^d$$
of sampling nodes in $\T^d$ \emph{rank\mbox{-}1 lattice}, cf. also~\cite{SlJo94}. Here the $\bmod\;\boldone$ means that we map each component of a vector to its fractional part. However, due to the periodicity of trigonometric polynomials the modulo part of the definition doesn't matter for the rest of the paper.
For given Fourier coefficients $\hat{p}_\boldk$, $\boldk\in J\subset\Z^d$ and $|J|<\infty$, the evaluation of a trigonometric polynomial $p\in\Pi(J)$ at all nodes of a rank-lattice $\Lambda(\boldz,M)$, i.e., the matrix vector product $\FourMat(\Lambda(\boldz,M),J)(\hat{p}_\boldk)_{\boldk\in J}$, simplifies to a univariate FFT of length $M$ due to
$$
p\left(\frac{j}{M}\boldz\right)=\sum_{\boldk\in J}\hat{p}_\boldk\e^{2\pi\ii\frac{j}{M}\boldk\cdot\boldz}=\sum_{l=0}^M\sum_{\boldk\cdot\boldz\equiv l\imod{M}}\hat{p}_\boldk
\e^{2\pi\ii\frac{j}{M}\boldk\cdot\boldz}=\sum_{l=0}^M\underbrace{\left(\sum_{\boldk\cdot\boldz\equiv l\imod{M}}\hat{p}_\boldk\right)}_{\hat{g}_l}
\e^{2\pi\ii\frac{jl}{M}}.
$$
The pre-calculation of all $\hat{g}_l$ requires a computational effort in $\OO{d|J|}$ and the subsequent univariate FFT in $\OO{M\log M}$, i.e., the evaluation requires $\OO{M\log M + d|J|}$ in total. The corresponding adjoint evaluation, i.e., an matrix vector product with the conjugate transpose of $\FourMat(\Lambda(\boldz,M),J)$, can also be realized using a univariate FFT and has the same computational complexity. Moreover, for sampling sets $\ST$ consisting of $K$ rank\mbox{-}1 lattices $\Lambda(\boldz_1,M_1),\dots,\Lambda(\boldz_K,M_K)$ , the computation of matrix vector products for matrices $\FourMat(\ST,J)$ and its conjugate transpose is realizable in
$\OO{\sum_{\ell=1}^K|M_\ell|\log M_\ell + Kd|J|}$ floating point operations, cf.~\cite{Kae16}, and is therefore highly efficient with regard to the size of the matrices.

We recall a basic fact about Fourier matrices with a rank\mbox{-}1 lattice as sampling scheme.

\begin{lemma}\label{lem:basic_aliasing}
Let $J\subset\Z^d$, $|J|<\infty$, be an index set and $\boldz\in\Z^d$ and $M\in\N$ generating vector and size of the rank\mbox{-}1 lattice $\Lambda(\boldz,M)$.
Then two columns $\bolda_{\boldh_1}$ and $\bolda_{\boldh_2}$, $\boldh_1,\boldh_2\in J$, of the matrix $\FourMat(\Lambda(\boldz,M),J)$
\begin{itemize}
\item either coincide, i.e., $\bolda_{\boldh_1}=\bolda_{\boldh_2}$ and $\boldh_1\cdot\boldz\equiv\boldh_2\cdot\boldz\imod{M}$,
\item or they are orthogonal, i.e., $\bolda_{\boldh_1}\cdot \bolda_{\boldh_2}=0$ and $\boldh_1\cdot\boldz\not\equiv\boldh_2\cdot\boldz\imod{M}$.
\end{itemize}
\end{lemma}
\begin{proof}
This is a simple consequence of the formula for the sum of the first $M$ terms of a geometric series, since 
$$\bolda_{\boldh_1}\cdot \bolda_{\boldh_2}=\sum_{j=0}^{M-1}\left(\e^{2\pi\ii\frac{(\boldh_1-\boldh_2)\cdot\boldz}{M}}\right)^j=\begin{cases}
0 & \frac{(\boldh_1-\boldh_2)\cdot\boldz}{M}\not\in\Z\\
M & \frac{(\boldh_1-\boldh_2)\cdot\boldz}{M}\in\Z
\end{cases}
=\begin{cases}
0 & \boldh_1\cdot\boldz\not\equiv\boldh_2\cdot\boldz\imod{M}\\
M & \boldh_1\cdot\boldz\equiv\boldh_2\cdot\boldz\imod{M}.
\end{cases}
$$
In the case $\boldh_1\cdot\boldz\equiv\boldh_2\cdot\boldz\imod{M}$, the columns $\bolda_{\boldh_1}$ and $\bolda_{\boldh_2}$ coincide.
\end{proof}

\begin{remark}\label{rem:aliasing_refined}
Due to the fact that two columns $\bolda_{\boldh_1}$ and $\bolda_{\boldh_2}$, $\boldh_1,\boldh_2\in \Mirror(I)$, $I\subset\N_0^d$, of the matrix $\FourMat(\Lambda(\boldz,M),\Mirror(I))$ are either orthogonal or identical and all columns $\bolda_{\boldh}$ are nonzero, we observe $\boldc_{\boldk}'\cdot\bolda_{\boldh}\neq 0$ for each $\boldk\in I$ and any $\boldh\in\Mirror(\{\boldk\})$. Here, $\boldc_{\boldk}'$, $\boldk\in I$, are the columns of the matrix $\FourMat(\ST,\Mirror(I))\ConMat(I)$, cf.\ also the proof of Lemma~\ref{lem:C_linear independent_columns}.
Taking Remark~\ref{rem:weakened_assumption} and Lemma~\ref{lem:C_linear independent_columns} into account, $\bolda_{\boldh}\not\in\sspan\{\,\bolda_{\boldl}\colon \boldl\in\Mirror(I\setminus\{|\boldh|\})\,\}$ immediately implies $\boldc_{|\boldh|}\not\in\sspan\{\boldc_{\boldl}\colon \boldl\in I\setminus\{|\boldh|\}\}$ if $\ST$ is a rank\mbox{-}1 lattice. Here, $\boldc_{\boldk}$, $\boldk\in I$, are the columns of the Chebyshev evaluation matrix $\ChebMat(\CosSet(\ST),I)$.
\end{remark}

Usually, distinct frequencies $\boldh_1,\boldh_2\in J\subset\Z^d$, $\boldh_1\neq\boldh_2$ and $|J|<\infty$, with coinciding matrix columns $\bolda_{\boldh_1}$ and $\bolda_{\boldh_2}$ are called aliasing frequencies. 
An undesirable reason for aliasing are inappropriate lattice sizes $M$.
We observe the implication
$$
\boldh_1\equiv\boldh_2\imod{M}\quad\Rightarrow\quad \boldh_1\cdot\boldz\equiv\boldh_2\cdot\boldz\imod{M}
$$
for each $\boldz\in\Z^d$. Here, the modulo at the left hand side is applied component-wise.
Accordingly, for a fixed lattice size $M$ we could have aliasing for each generating vector $\boldz\in\Z^d$.
As a consequence, for theory we avoid lattice sizes $M$ for which there exist $\boldh_1,\boldh_2\in J$, $\boldh_1\neq\boldh_2$, such that
$\boldh_1\equiv\boldh_2\imod{M}$ holds, since the sampling values at all lattices of such sizes $M$ do not allow for distinguishing the two monomials $\e^{2\pi\ii\boldh_1\cdot}$
and $\e^{2\pi\ii\boldh_2\cdot}$.
For that reason, we introduce the set
\begin{equation}
J_{\bmod M}:=\{\boldh\bmod M\colon \boldh\in J\}\subset[0,M-1]^d\cap\Z^d\,,
\label{eq:set_mod_M}
\end{equation}
and observe that $|J|=|J_{\bmod M}|<\infty$ implies that there do not exist $\boldh_1,\boldh_2\in J$, $\boldh_1\neq\boldh_2$, with
$\boldh_1\equiv\boldh_2\imod{M}$. In addition, we introduce the number $N_J$ of the occurring maximal powers of the monomials within $\Pi(J)$ or $\ChebPol(\abs(J))$ by
\begin{equation}
N_J:=\max_{\boldh\in J} \|\boldh\|_\infty.
\label{eq:max_occur_powers}
\end{equation}
In fact, for $M\in\N$ and $M>2N_J$ one observes $|J|=|J_{\bmod M}|$.
Moreover, we observe $N_I=N_{\Mirror(I)}$ and thus $|\Mirror(I)|=|\Mirror(I)_{\bmod M}|$ for $I\subset\N_0^d$ and $M\in\N$ with $M>2N_I$.

\section{Joining suitable sampling sets for discretization}\label{sec:JoiningSamplingSets}

According to Lemma~\ref{lem:span_subset}, one approach for constructing spatial discretizations of $\ChebPol(I)$ is the successive determination of
sets $\SX_i\subset[-1,1]^d$, $|\SX_i|<\infty$, $i=1,2,\dots$, of sampling nodes. For each sampling set $\SX_i$ one determines all the columns, i.e., indices $\boldk\in I$, of $\ChebMat(\SX_i,I)$ which are not in the span of the other columns. Once for each of the indices $\boldk\in I$ the corresponding column is not in the span of all others for at least one $\ChebMat(\SX_i,I)$, the matrix $\ChebMat(\bigcup_i \SX_i,I)$ is of full column rank.
Obviously, the sampling sets $\SX_i$ must be selected appropriately.

\begin{remark}
The aforementioned strategy provides a general approach for discretizing finite dimensional function spaces. Assuming a linear independent set
of basis functions is given, one can build up the discretization step by step. Clearly, the crucial challenge
is that the chosen sampling sets must be selected such that there is at least a chance that each column might be not in the span of all others.
\end{remark}

An alternative approach is to determine a sampling set $\ST\subset\T^d$ which fulfills Corollary~\ref{cor:matrix_C_full_crank} in a similar way. One chooses $\ST_i\subset\T^d$, $|\ST_i|<\infty$, $i=1,2,\dots$, determines all columns, i.e., $\boldh\in\Mirror(I)$, of $\FourMat(\ST_i,\Mirror(I))$ which are not in the span of all other columns of $\FourMat(\ST_i,\Mirror(I))$. The corresponding indices $\boldh$ are collected in a set $\tilde{J}$. Once $I\subset \abs(\tilde{J})$ is fulfilled, i.e., $I=\abs(\tilde{J})$, the matrix 
$\ChebMat(\SX,I)\triangleq\FourMat(\bigcup_i\ST_i,\Mirror(I))\ConMat(I)$, $\SX=\CosSet(\bigcup_{i}\ST_i):=\bigcup_{i}\{\cos(2\pi \boldt)\colon \boldt\in\ST_i\}$, has full column rank due to Corollary~\ref{cor:matrix_C_full_crank} and Lemma~\ref{lem:span_subset}. Again, the chosen sampling sets $\ST_i$ must be suitably selected. Well-chosen rank\mbox{-}1 lattices have such good properties as the theoretical considerations in Section~\ref{sec:theory} verify.

\subsection*{Additional improvement}\label{sec:add_impro_general}

An additional refinement of the aforementioned strategy promises much lower
cardinalities of the constructed spatial discretizations.
Once the first well-chosen sampling set $\SX_1$ is fixed, it is clear that all columns of the sampling matrix $\ChebMat(\SX_1,I)$ which fulfill the \NISOR condition will fulfill the \NISOR condition w.r.t.\ $\ChebMat(\SX,I)$ for each sampling set $\SX\supset\SX_1$ as well, cf.\ 
Lemma~\ref{lem:span_subset}. In other words, using an approach that 
builds up a discretization by means of unions of sampling sets as described above, one needs to identify each column as \NISOR only once. Accordingly, it is possible to reduce the size of the discretization problem after each step.

\subsubsection*{Immediate improvement}

According to Lemma~\ref{lem:matrices}, we can further improve the discretization approach in the following way.
For a given sampling set $\mathcal{X}_1$ we categorize the columns of $\ChebMat(\mathcal{X}_1,I)$ in those which fulfill the \NISOR condition and collect them in $\AMat_{11}$ and the others in $\AMat_{12}$. We denote the set of indices $\boldk\in I$ which belong to $\AMat_{11}$  by $J_1$.
Subsequently, we select a second sampling set $\mathcal{X}_2$ and consider the matrix $\ChebMat(\mathcal{X}_2,I\setminus J_1)$ in the role of $\AMat_{22}$. We determine the columns of $\ChebMat(\mathcal{X}_2,I\setminus J_1)$ which fulfill the \NISOR condition w.r.t.\ $\ChebMat(\mathcal{X}_2,I\setminus J_1)$ and denote the corresponding indices by $J_2$. According to Lemma~\ref{lem:matrices}, each of the first $|J_1|+|J_2|$ columns of the matrix 
\begin{equation}
\begin{pmatrix}
\ChebMat(\mathcal{X}_1,J_1) & \ChebMat(\mathcal{X}_1,J_2) & \ChebMat(\mathcal{X}_1,I\setminus (J_1\cup J_2))\\
\ChebMat(\mathcal{X}_2,J_1) & \ChebMat(\mathcal{X}_2,J_2) & \ChebMat(\mathcal{X}_2,I\setminus (J_1\cup J_2))
\end{pmatrix}\label{eq:chebmatrix:cherry_pick}
\end{equation}
fulfill the \NISOR condition w.r.t.\ the matrix in \eqref{eq:chebmatrix:cherry_pick}. Note that\eqref{eq:chebmatrix:cherry_pick} is essentially the same as $\ChebMat(\mathcal{X}_1\cup \mathcal{X}_2, I)$ up to column permutations. Consequently,
$$\begin{pmatrix}
\ChebMat(\mathcal{X}_1,J_1) & \ChebMat(\mathcal{X}_1,J_2)\\
\ChebMat(\mathcal{X}_2,J_1) & \ChebMat(\mathcal{X}_2,J_2)
\end{pmatrix}
$$
can
be used as $\AMat_{11}$ in the next step.
Now, one continues in selecting suitable sampling sets $\mathcal{X}_\ell$ up to $\ell=K$ for which $J_1\cup\dots\cup J_K=I$ is fulfilled.

\subsubsection*{Improvement via Fourier}

The strategy explained in the last paragraph can also be applied to the matrices $\FourMat(\ST,\Mirror(I))$ with some slight modifications.
For a given sampling set $\ST_1\in\T^d$ we categorize the columns of
$\FourMat(\ST_1,\Mirror(I))$ in those which fulfill the \NISOR condition w.r.t.\ $\FourMat(\ST_1,\Mirror(I))$ and the others. We denote the set of indices $\boldh\in \Mirror(I)$ which belong to the columns that fulfill the \NISOR condition by $\tilde{J}_1$. Applying Lemma~\ref{lem:C_linear independent_columns}, each column of the matrix $\ChebMat(\SX_1,J_1)$, $\mathcal{X}_1=\CosSet(\ST_1)$ and $J_1=\abs(\tilde{J}_1)$, fulfills the \NISOR condition w.r.t.\ $\ChebMat(\SX_1,I)$. Accordingly, we just have to consider the index set $I\setminus J_1$, i.e., it remains to determine a sampling set $\tilde{\SX}$ such that the matrix 
$\ChebMat(\tilde{\SX},I\setminus J_1)$ is of full column rank. To this end, one can determine $\tilde{\ST}\subset\T^d$ such that $\FourMat(\tilde{\ST},\Mirror(I\setminus J_1))$ fulfills the assumptions of Corollary~\ref{cor:matrix_C_full_crank}. Then the set $\CosSet(\tilde{\ST}\cup\ST_1)$ is a spatial discretization of $\ChebPol(I)$.

Of course, the aforementioned strategy can be applied iteratively as long as there are columns that have not yet been identified as %
\NISOR{}\!.
It should be noted that at a certain stage of such an iterative approach, there may be columns in the matrix $\ChebMat(\CosSet(\SX_1\cup\dots\cup\SX_l),I)$, $l\ge 2$, that have not yet been identified fulfilling the \NISOR condition, but which, due to the combination of the sampling sets, already satisfy this condition.

\section{Theory on (cosine transformed) multiple rank\mbox{-}1 lattices}\label{sec:theory}

The goal of this section is to combine Corollary~\ref{cor:matrix_C_full_crank} with known results on rank\mbox{-}1 lattices to obtain theoretical estimates on the number of sampling nodes that a discretization constructed using the proposed approaches contains and to estimate the associated failure probability. We refer to \cite[Section~3]{Kae17} from which we adapt the proofs in order to obtain specific results for our present purposes.
In contrast to the general formulation in Corollary~\ref{cor:matrix_C_full_crank}, we restrict the considerations
to a special case. More precisely, for a given index set $I\subset\N_0^d$, we construct a sampling set $\mathcal{T}\subset\T^d$ such that for each $\boldk\in I$ the column $\bolda_{\boldk}$ of $\FourMat(\ST,\Mirror(I))$ fulfills the \NISOR condition, i.e., $\bolda_{\boldk}\not\in\sspan\{\bolda_{\boldl}\colon \boldl\in\Mirror(I)\setminus\{\boldk\}\}$. Obviously, this is even more restrictive than the assumptions in Corollary~\ref{cor:matrix_C_full_crank}.
In fact, avoiding (too much) aliasing frequencies is the main idea of the following construction approach.

\begin{lemma}\label{lem:count_aliasing_probability_cheb}
We fix a frequency $\zb k\in I\subset\N_0^d$, $|I|<\infty$, and a prime number $M$ such that $|\Mirror(I)_{\bmod M}|=|\Mirror(I)|$, cf. \eqref{eq:set_mod_M}. In addition, we choose a generating vector $\boldz\in[0,M-1]^d$ at random. Then, with probability of at most
$\frac{|\Mirror(I)|-1}{M}$ the frequency $\zb k$ aliases to at least one other frequency $\zb l\in \Mirror(I)\setminus\{\zb k\}$.
\end{lemma}
\begin{proof}
Taking $I\subset\Mirror(I)$ and, consequently, $\boldk\in\Mirror(I)$ into account, we apply \cite[Lemma 3.1]{Kae17} to $\boldk$ as element of $\Mirror(I)\subset\Z^d$.
\end{proof}

\begin{theorem}\label{thm:gen_mr1l_cheb}
We consider the given frequency set $I\subset\N_0^d$, $|I|<\infty$, and the corresponding mirrored frequency set $\Mirror(I)$ and we fix an element $\zb k\in I$.
In addition,  we fix $c>1$ and $\delta\in(0,1]$ and we determine two numbers
\begin{align}
\lambda&\ge c(|\Mirror(I)|-1),\label{eq:def_lambda1}\\
L&\ge\ceil{\left(\frac{c}{c-1}\right)^2\frac{\ln |I|-\ln \delta}{2}}\label{def:s}
\end{align}
and a prime number $M\ge \lambda$ such that $|\Mirror(I)_{\bmod M}|=|\Mirror(I)|$ holds.

Moreover, we draw $L$ generating vectors $\boldz_\ell\in [0,M-1]^d$, $\ell=1,\dots,L$, i.i.d.\ uniformly at random.
Then, the probability that the frequency $\zb k\in I$ aliases to any other frequency within $\Mirror(I)$ for each rank\mbox{-}1 lattice $\Lambda(\boldz_\ell,M_\ell)$ is bounded from above by $\e^{-2L\left(\frac{c-1}{c}\right)^2}\le \frac{\delta}{|I|}$.
\end{theorem}
\begin{proof}
For the fixed frequency $\zb k\in I$, we define the random variables
$$
Y_\ell^{\zb k}:=\begin{cases}
0&\colon \zb k \textnormal{ does not alias to another frequency within $\Mirror(I)$ using $\Lambda(\boldz_\ell,M)$},\\
1&\colon \zb k \textnormal{ aliases to at least one other frequency within $\Mirror(I)$ using $\Lambda(\boldz_\ell,M)$}.\\
\end{cases}
$$
The random variables $Y_\ell^{\zb k}$, $\ell=1,\ldots, L$, are independent and identically distributed with a specific mean $\mu\le\frac{|\Mirror(I)|-1}{M}\le\frac{|\Mirror(I)|-1}{\lambda}\le\frac{1}{c}$, due to Lemma \ref{lem:count_aliasing_probability_cheb}.
Hoeffding's inequality \cite{Hoeff63} allows for the estimate
\begin{align*}
\mathcal{P}\left\{\sum_{\ell=1}^LY_\ell^{\zb k}=L\right\} &=\mathcal{P}\left\{L^{-1}\sum_{\ell=1}^LY_\ell^{\zb k}-\mu=1-\mu\right\}\le\mathcal{P}  \left\{L^{-1}\sum_{\ell=1}^LY_\ell^{\zb k}-\mu\ge 1-\varepsilon-\mu\right\}\\
&\le \e^{-2L\left(1-\varepsilon-\mu\right)^2}= \e^{-2L\left(\frac{c-1}{c}\right)^2}
\le\e^{\ln \delta-\ln |I|}=\frac{\delta}{|I|}
\end{align*}
for the specific choice $\varepsilon=\frac{1}{c}-\mu\ge0$.
\end{proof}

Let us consider the matrix $\FourMat\left(\ST,\Mirror(I)\right)$, where $\ST=\bigcup_{\ell=1}^L\Lambda(\boldz_\ell,M)$ is constructed as Theorem~\ref{thm:gen_mr1l_cheb} suggests. Taking Lemmas~\ref{lem:basic_aliasing} and~\ref{lem:matrices}
into account, the probability that the $\boldk$th column $\bolda_{\boldk}$, $\boldk\in I$, of $\FourMat\left(\ST,\Mirror(I)\right)$ fulfills the \NISOR condition is at least $1-\delta/|I|$.
Accordingly, the probability that the $\boldk$th column $\boldc_{\boldk}$ of $\ChebMat\left(\CosSet(\ST),I\right)$ fulfills the \NISOR condition is at least $1-\delta/|I|$ due to Lemma~\ref{lem:C_linear independent_columns}.

Another simple conclusion allows for the estimate on the probability
that each column of $\ChebMat\left(\CosSet(\ST),I\right)$ fulfills the \NISOR condition, i.e., the matrix $\ChebMat\left(\CosSet(\ST),I\right)$ is of full column rank and, thus, $\CosSet(\ST)$ is a spatial discretization of $\ChebPol(I)$.

\begin{theorem}\label{thm:prob_bound_T_cheb}
Choosing $M$, $L$, and $\Lambda(\boldz_\ell,M)$, $\ell=1,\ldots,L$, as stated in Theorem \ref{thm:gen_mr1l_cheb}, the probability that we can uniquely reconstruct all Chebyshev coefficients from the sampling values of $P\in\operatorname{C\Pi}_I$ at the cosine transformed multiple rank\mbox{-}1 lattice $\CosSet\left(\bigcup_{\ell=1}^L \Lambda(\boldz_\ell,M)\right)=\bigcup_{\ell=1}^L\{\cos(2\pi \boldt)\colon \boldt\in \Lambda(\boldz_\ell,M)\}$ is bounded from below by $1-\delta$.
\end{theorem}
\begin{proof}
Applying the union bound on the probability determined in Theorem~\ref{thm:gen_mr1l_cheb} yields an upper
bound on the failure probability
$$
1-\mathcal{P}\left(\bigcap_{\zb k\in I}\left\{\sum_{\ell=1}^LY_\ell^{\zb k}<L\right\}\right)=\mathcal{P}\left(\bigcup_{\zb k\in I}\left\{\sum_{\ell=1}^LY_\ell^{\zb k}=L\right\}\right)\le
\sum_{\zb k\in I}\mathcal{P}\left(\sum_{\ell=1}^LY_\ell^{\zb k}=L\right)\le  |I|\frac{\delta}{|I|}\,.
$$
Further, we apply Corollary~\ref{cor:matrix_C_full_crank} which yields the assertion.
\end{proof}

The difference to the results in \cite[Section~3]{Kae17} lies in the columns that must have the \NISOR condition. For a spatial discretization of $\Pi(\Mirror(I))$, all columns of the matrix $\FourMat(\ST,\Mirror(I))$ must satisfy the \NISOR condition. According to \cite[Section~3]{Kae17}, this is guaranteed with probability of at least $\delta$ if the sampling set $\ST$ consists of at least $\ceil{\left(\frac{c}{c-1}\right)^2\frac{\ln |\Mirror(I)|-\ln \delta}{2}}$ randomly drawn rank\mbox{-}1 lattices of size $M$ as above.

In contrast, a set of cosine transformed rank\mbox{-}1 lattices is a spatial discretization of $\ChebPol(I)$ if the \NISOR condition is satisfied for a subset of columns of $\FourMat(\ST,\Mirror(I))$. This is already guaranteed with probability at least $\delta$ for a lower number, namely $\ceil{\left(\frac{c}{c-1}\right)^2\frac{\ln |I|-\ln \delta}{2}}$, of randomly drawn rank\mbox{-}1 lattices of the same size $M$.

\begin{remark}\label{rem:remark_cards_cheb_lattices}
The number of sampling nodes within a single cosine transformed rank\mbox{-}1 lattice, i.e., within the set $\CosSet(\Lambda(\boldz,M))=\{\cos(2\pi \boldt)\colon \boldt\in \Lambda(\boldz,M)\}$, is bounded from above by $\frac{M+1}{2}$, due to the fact that
$$
\cos\left(2\pi \frac{j}{M}\boldz\right)=\cos\left(-2\pi \frac{j}{M}\boldz\right)=
\cos\left(2\pi\frac{M}{M}\boldz -2\pi \frac{j}{M}\boldz\right)=\cos\left(2\pi \frac{M-j}{M}\boldz\right)
$$
holds for each $j\in\N_0$, cf.\ also \cite{SuNuCo14,KuoMiNoNu19}. Since $\boldzero\in\T^d$ is an element of each rank\mbox{-}1 lattice, each cosine transformed rank\mbox{-}1 lattice contains the sampling node $(1,\dots,1)^\top\in[-1,1]^d$.

Moreover, for fixed $c=2$ and given $I\subset\N_0^d$, $1\le|I|<\infty$ with $|\Mirror(I)|>1$, we determine $$M:=\operatorname{nextprime}(\max(2(\Mirror(I)-1),2N_I)),$$ which implies $|\Mirror(I)_{\bmod M}|=|\Mirror(I)|$, cf.~\eqref{eq:max_occur_powers} and the following paragraph. Consequently, $M$ fulfills the requirements of Theorem~\ref{thm:prob_bound_T_cheb} with $c=2$. Then, $M$ is bounded from above by
$$
M
\le 2 (2\max(|\Mirror(I)|-1,N_I))-1
$$
due to Bertrand's postulate. As a result, we observe that
$
|\CosSet(\Lambda(\boldz,M))|
\le 2\max(\Mirror(I)-1,N_I)$ holds.
In addition, for $r\ge 0$, we fix $\delta=|I|^{-r}$ and $L=\lceil 2(r+1)\ln|I|\rceil$.
For the aforementioned fixed parameter choice, the sampling sets constructed by Theorem~\ref{thm:prob_bound_T_cheb} have cardinalities bounded
by $2\,\lceil 2(r+1)\ln|I|\rceil \max(|\Mirror(I)|-1,N_I)$ and the matrix $\ChebMat(\SX,I)$, $\SX=\bigcup_{\ell=1}^L\CosSet(\Lambda(\boldz_\ell,M))$, is of full column rank with probability at least $1-|I|^{-r}$.
\end{remark}

\section{Improvements for practical implementation}\label{sec:improvements}

The analysis in Section~\ref{sec:theory} determines requirements for a sampling set $\mathcal{T}$ consisting of multiple rank\mbox{-}1 lattices such that each column of the matrix $\FourMat(\ST,\Mirror(I))$ which belongs to some $\boldk\in I\subset\Mirror(I)$ has the \NISOR condition with a certain probability. A subsequent application of Corollary~\ref{cor:matrix_C_full_crank}
 leads to the statement about the Chebyshev matrices. However, the requirements of Corollary~\ref{cor:matrix_C_full_crank} to the matrix $\FourMat(\ST,\Mirror(I))$ are slightly weaker, which leads to the insight, that $\FourMat(\ST,\Mirror(I))$ fulfills these requirements when constructing the sampling set $\ST$ in the way Theorem~\ref{thm:prob_bound_T_cheb} suggests with even higher probability. From another point of view, the discretization property might even be fulfilled for, e.g., smaller lattice sizes $M$ or lower numbers $L$ of lattices with the estimated probability.

Section~\ref{sec:constr_wo_index_set} shortly explains, how to determine spatial discretizations without dealing (too much) with index sets.
Moreover, Section~\ref{sec:constr_w_index_set} discusses several approaches for determining spatial discretizations of $\ChebPol(I)$.
Several improvements aim to significantly reduce the number of sampling nodes.

\subsection{Construction without determining index sets}\label{sec:constr_wo_index_set}

The suggested construction of multiple rank\mbox{-}1 lattices in Theorem~\ref{thm:prob_bound_T_cheb}
just uses the cardinality of the index set $I$ and the cardinality of the mirrored index set $\Mirror(I)$. If the maximal extent $N_{\Mirror(I)}$ of $\Mirror(I)$ is additionally bounded by its cardinality $|\Mirror(I)|$ (in practical applications the rule rather than the exception) then the discretization can be constructed without knowing the index sets itself. It is enough to have some close upper bounds on $|I|$ and $|\Mirror(I)|$.

In any case, the construction from the known index set $I$ without determining the mirrored index set $\Mirror(I)$ is a simple task.

\begin{example}[Construction without $\Mirror(I)$]
We assume that $I\subset\N_0^d$ is given and we fix $c=2$ and $r\ge 0$. We determine
the cardinality of $\Mirror(I)$ merely.
\begin{align*}
|\Mirror(I)|&=\sum_{\boldk\in I}2^{\|\boldk\|_0}\\
M&=\operatorname{nextprime}(\max(2(|\Mirror(I)|-1),2N_I)),\quad\text{cf.~\eqref{eq:max_occur_powers}}\\
L&=\lceil 2(1+r)\ln|I|\rceil
\end{align*}
Subsequently, we draw $L$ i.i.d.\ vectors uniformly at random from $[0,M-1]^d$ and denote them by $\boldz_\ell$, $\ell=1,\dots,L$.
Then the sampling set $\bigcup_{i=1}^L\CosSet\left(\Lambda(\boldz_\ell,M)\right)
$ is a spatial discretization of $\ChebPol(I)$ with probability at least $1-|I|^{-r}$. Assuming that $N_I\lesssim |\Mirror(I)|$, the number of distinct sampling nodes can be estimated by
$$
\bigcup_{i=1}^L\CosSet\left(\Lambda(\boldz_\ell,M)\right)
\le 1+L\frac{M-1} {2}
\in \OO{|\Mirror(I)|\ln|I|}\subset \OO{2^d|I|\ln|I|}\,.
$$
The computational costs of that approach are very manageable in $\OO{d|I|}$.
\end{example}

Indeed, the construction approaches mentioned in this section do not guarantee
success, since there is still a certain failure probability. However, knowing
the index sets allows for simply checking the discretization properties of the constructed sampling set
or checking conditions that imply the discretization property.

\subsection{Construction with determined index sets}\label{sec:constr_w_index_set}

Theorem~\ref{thm:prob_bound_T_cheb} allows for the estimate that with probability at least
$1-|I|\e^{-2L\left(\frac{c-1}{c}\right)^2}$, $c>1$, a number of $L$ randomly drawn rank\mbox{-}1 lattices of the determined (prime) size $M\ge\lambda$ as in \eqref{eq:def_lambda1}, constitutes a spatial discretization of 
$\ChebPol(I)$.
Moreover, once the $L$ rank\mbox{-}1 lattices are drawn, one can easily check whether or not each index $\boldk\in I$ has some mirror in $\boldh\in\Mirror(\{\boldk\})$ whose corresponding column in at least one $\FourMat(\Lambda(\boldz_\ell,M),\Mirror(I))$, $\ell=1,\dots,L$, fulfills the \NISOR condition. 
If each $\boldk\in I$ has such a column in at least one $\FourMat(\Lambda(\boldz_\ell,M),\Mirror(I))$, then the set $\CosSet\left(\bigcup_{\ell=1}^L\Lambda(\boldz_\ell,M)\right)$
is a spatial discretization of $\ChebPol(I)$.
In fact, we compute the values $\boldh\cdot\boldz_\ell\bmod{M}$ for all $\boldh\in\Mirror(I)$ and all $\ell\in\{1,\dots,L\}$ and we determine the sets
\begin{equation}
J_\ell:=\{|\boldk|\colon \boldk\in \Mirror(I),\; \not\exists \boldh\in \Mirror(I)\setminus\{\boldk\}\text{ with }\boldk\cdot \boldz_{\ell}\equiv\boldh\cdot \boldz_{\ell}\imod{M}\}\label{eq:def_Jl}
\end{equation}
 due to Lemma~\ref{lem:basic_aliasing}.
If $\bigcup_{\ell=1}^L J_\ell=I$ holds, $\CosSet\left(\bigcup_{\ell=1}^L\Lambda(\boldz_\ell,M)\right)$ is guaranteed to be a spatial discretization of $\ChebPol(I)$.

Please note that it is not necessary to construct the mirrored index set. In fact, holding an integer vector of size $|\Mirror(I)|$ is enough in order to save all inner products $\boldh\cdot\boldz_\ell\bmod{M}$, $\boldh\in\Mirror(I)$, for fixed $\ell$. All the inner products can be computed from $I$ directly. Moreover, the sets $J_\ell$ need not to be physically stored since it is enough to store a pointer to the corresponding elements in $I$. Nevertheless, the computational complexity is in $\OO{|\Mirror(I)|(\log|I|+d)L}$, cf. Algorithm~\ref{alg:determine_J_ell}, which already includes the random drawing of the generating vectors.
Further details on the implementation for the trigonometric system can be found in \cite{Kae17}.

\begin{algorithm}[tb]
\caption{Determining the sets $J_\ell$}
\label{alg:determine_J_ell}
  \begin{tabular}{p{1.4cm}p{2.05cm}p{10.8cm}}
    Input: 	& $I\subset\N_0^d$ 	& frequency set\\
    		& $L$	& number of rank\mbox{-}1 lattices to be considered\\
    		& $M$	& prime number
  \end{tabular}
		
  \begin{algorithmic}[1]
	\For{$\ell=1,\dots,L$}
		\State draw $\boldz_{\ell}$ from $[0,M-1]^d\cap\Z^d$ uniformly at random
	    \State determine $J_\ell:=\{|\boldk|\colon \boldk\in \Mirror(I),\; \not\exists \boldh\in \Mirror(I)\setminus\{\boldk\}\text{ with }\boldk\cdot \boldz_{\ell}\equiv\boldh\cdot \boldz_{\ell}\imod{M}\}$
	\EndFor
  \end{algorithmic}
  \begin{tabular}{p{1.75cm}p{4.05cm}p{8.49cm}}
    Output: & $\boldz_{1},\dots,\boldz_{L}$ & generating vectors\\
	    & $J_{1},\dots,J_L$ & index sets
\\    \cmidrule{1-3}
	\multicolumn{3}{l}{Complexity:
    $\OO{|\Mirror(I)|(\log|\Mirror(I)|+d)L}=\OO{|\Mirror(I)|(\log|I|+d)L}$}
  \end{tabular}
\end{algorithm}

\begin{remark}
Applying Remark~\ref{rem:aliasing_refined}, we can slightly modify $J_\ell$, cf.\ \eqref{eq:def_Jl}, to
$$\mathring{J}_\ell:=\{|\boldk|\colon \boldk\in \Mirror(I),\; \not\exists \boldh\in \Mirror(I\setminus\{|\boldk|\})\text{ with }\boldk\cdot \boldz_{\ell}\equiv\boldh\cdot \boldz_{\ell}\imod{M}\}\,.$$
Note that $J_\ell\subset\mathring{J}_\ell$ holds and, thus, when using $\mathring{J}_\ell$ instead of $J_\ell$, (a few) more \NISOR condition fulfilling columns of the considered matrices might be identified in each step. As a consequence, the following theoretical results remain valid even when using $\mathring{J}_\ell$.
Note that the calculation of $\mathring{J}_\ell$ is somewhat more complicated than the calculation of $J_\ell$.
\end{remark}

\subsubsection{Greedy improvement}\label{sssec:greedy_improvement}

Applying the last mentioned strategy, an additional simple and affordable improvement immediately catches the eye.
The order of the randomly drawn rank\mbox{-}1 lattices might affect the number of used rank\mbox{-}1 lattices since there are 
at least slight differences in the number of columns in $\FourMat(\Lambda(\boldz,M),\Mirror(I))$ which are \NISOR depending on $\boldz$.
Accordingly, it might be beneficial to determine the column indices of 
$\FourMat(\Lambda(\boldz_\ell,M),\Mirror(I))$ which have the \NISOR condition for each $\ell=1,\dots,L$ first.
We denote the corresponding indices of these columns by $\tilde{J}_\ell\subset\Mirror(I)$ and save the sets $J_\ell:=\abs(\tilde{J}_\ell)$.
Subsequently, we apply a greedy 
approach. The strategy is to pick that rank\mbox{-}1 lattice $\Lambda(\boldz_\ell,M)$ which yields most columns $c_\boldk$, $\boldk\in I$, fulfilling the \NISOR condition w.r.t.\ $\ChebMat(\CosSet(\Lambda(\boldz_\ell,M)),I)$ to fix $\SX_1$. In other words, we determine $\argmax\{|J_\ell|\}:=\{\ell\in\{1,\dots,L\}\colon |J_\ell|=\max_{j=1,\dots,L}|J_j|\}$, pick one $\ell_1\in \argmax\{|J_\ell|\}$ and set $\SX_1=\CosSet(\Lambda(\boldz_{\ell_1},M))$ as well as $J^{(1)}=J_{\ell_1}$.
Subsequently, we update $J_\ell:=J_\ell\setminus J^{(1)}$, $\ell=1,\dots,L$, pick $\ell_2\in\argmax\{|J_\ell|\}$, which is the index of that rank\mbox{-}1 lattice with most \NISOR columns w.r.t.\ $\ChebMat(\CosSet(\Lambda(\boldz_\ell,M)),I\setminus J^{(1)})$ and set 
$\SX_2=\CosSet(\Lambda(\boldz_{\ell_2},M))$ as well as $J^{(2)}=J_{\ell_2}$.

We continue reducing $J_\ell$ by updating $J_\ell:=J_\ell\setminus J^{(2)}$, $\ell=1,\dots,L$, and picking $\ell_3\in\argmax\{|J_\ell|\}$ and so on until 
$\max_{\ell=1,\dots,L}|J_\ell|=0$. Clearly, this might be the case, before all $L$ rank\mbox{-}1 lattices are chosen, i.e., at some point where only $\SX_1$,\dots,$\SX_{L'}$, $L'<L$, are used in order to obtain $\max_{\ell=1,\dots,L}|J_\ell|=0$.
In that way, we select the most contributing rank\mbox{-}1 lattices under consideration in each step and avoid less suitable ones.

After execution of the described algorithm, there are two possibilities:
\begin{itemize}
\item $\bigcup_{k=1}^{L'}J^{(k)}=I$,
\item $\bigcup_{k=1}^{L'}J^{(k)}\subsetneq I$.
\end{itemize}
In the first case, the sampling set $\CosSet\left(\bigcup_{k=1}^{L'}\SX_k\right)$ is a spatial discretization of $\ChebPol(I)$. In the second case,
it is not guaranteed that $\CosSet\left(\bigcup_{k=1}^{L'}\SX_k\right)$ is a spatial discretization of $\ChebPol(I)$.
Of course, it might be one, but to check this requires additional computational effort of higher complexity compared to the calculations performed so far.
Accordingly, a trial and error strategy might be the better option.

Note in the case $\bigcup_{k=1}^{L'}J^{(k)}\subsetneq I$ even $\CosSet\left(\bigcup_{\ell=1}^{L}\Lambda(\boldz_\ell,M)\right)$ might not be a spatial discretization of $\ChebPol(I)$. 
In the case $\bigcup_{k=1}^{L'}J^{(k)}=I$, we call the greedy approach successful.

Clearly, calculating $\argmax\{|J_\ell|\}$ and updating the sets $J_\ell$ requires additional computational effort in each step, which in any case causes at most $\OO{|I| L}$ (not necessarily floating point) operations. Since we achieve at least one additional $\ell$ with $J_\ell=\emptyset$ at the end of each step, we have at most $L$ steps. Consequently, the additional computational effort is in $\OO{|I| L^2}$, i.e., the total computational effort is in
$$
\OO{|\Mirror(I)|(\log|I|+d)L+|I|L^2},
$$
i.e., for $L\lesssim\log|I|$, cf.~\eqref{def:s}, we observe a total complexity in
$\OO{|\Mirror(I)|(\log|I|+d)\log|I|}$.

\subsubsection{Improvement taking Section~\ref{sec:add_impro_general} into account}\label{sssec:improvement_simple}

We assume $|I|\ge 2$ and we fix $\delta\in(0,1]$.
Due to Theorem~~\ref{thm:gen_mr1l_cheb}, at least one of $L:=\lceil 2(\ln|I|-\ln\delta)\rceil$ randomly drawn generating vectors $\boldz\in[0,M-1]^d\cap\Z^d$, $M$ prime and $M > 2\max(|\Mirror(I)|-1, N_I)$, provides a sampling set $\CosSet(\Lambda(\boldz,M))$ such that at least $|I|/L$ columns of $\ChebMat(\CosSet(\Lambda(\boldz,M)),I)$ fulfill the \NISOR condition with probability at least $1-\delta$. Due to Section~\ref{sec:add_impro_general} these columns fulfill the \NISOR condition w.r.t.\ each $\ChebMat(\SX,I)$ provided that $\SX\supset\CosSet(\Lambda(\boldz,M))$ holds.

Consequently, the columns which are already known to fulfill the \NISOR condition do not need to be checked again for this condition. In other words, the remaining part of the task is to construct a sampling set $\SX'$ such that  $\ChebMat(\SX',I')$ has full column rank, where $I'$ is the index set $I$ reduced by those column indices of $\ChebMat(\CosSet(\Lambda(\boldz,M)),I)$ for which the \NISOR condition w.r.t.\ $\ChebMat(\CosSet(\Lambda(\boldz,M)),I)$ holds.

Accordingly, the following strategy yields a spatial discretization of $\ChebPol(I)$. We determine $M_0 =\operatorname{nextprime}(2\max(|\Mirror(I)|-1,N_I))$ and draw a set of $L=\lceil 2(\ln|I|-\ln\delta)\rceil$ generating vectors $\boldz_\ell^{(0)}\in[0,M_0-1]^d$ uniformly at random. For each $\ell\in\{1,\dots,L\}$, we determine the sets $\tilde{J}_\ell^{(0)}\subset\Mirror(I)$ which collect the indices of $\FourMat(\Lambda(\boldz_\ell^{(0)},M_0),\Mirror(I))$ that fulfill the \NISOR condition. Subsequently, we fix $\SX_0$ which is the best of the $L$ cosine transformed rank\mbox{-}1 lattices.

We choose the 
cardinality of $J_\ell^{(0)}=\abs(\tilde{J}_\ell^{(0)})$ as quality criteria in order to grade the $L$ randomly drawn rank\mbox{-}1 lattices $\Lambda(\boldz_\ell^{(0)},M_0)$, because the larger the $|J_\ell^{(0)}|$ is, the more the index set $I$ will be reduced in the following steps.
Similar as in Section~\ref{sssec:greedy_improvement}, we determine $\argmax\{|J_\ell^{(0)}|\}:=\{\ell\in\{1,\dots,L\}\colon |J_\ell^{(0)}|=\max_{j=1,\dots,L}|J_j^{(0)}|\}$, pick one $\ell_0\in \argmax\{|J_\ell^{(0)}|\}$, fix $J^{(0)}_{\max}=J_{\ell_0}^{(0)}$ and 
$\SX_0=\CosSet\left(\Lambda(\boldz_{\ell_0}^{(0)},M_0)\right)$.

Subsequently, we restart the just mentioned process for $I_1=I\setminus J_{\max}^{(0)}$ with fixed $L:=\lceil 2(\ln|I|-\ln\delta)\rceil$ in order to determine $\SX_1$.
In detail, we determine a prime number $M_1=\operatorname{nextprime}(2\max(|\Mirror(I_1)|-1,N_{I_1}))$, draw $L$ generating vectors $\boldz_\ell^{(1)}\in[0,M_1-1]^d$ uniformly at random, determine $\tilde{J}_\ell^{(1)}\subset\Mirror(I_1)$ which collect the indices of $\FourMat(\Lambda(\boldz_\ell^{(1)},M_1),\Mirror(I_1))$ that fulfill the \NISOR condition, compute  $J_\ell^{(1)}=\abs(\tilde{J}_\ell^{(1)})$, $\ell=1,\dots,L$, and pick one $\ell_1\in \argmax\{|J_\ell^{(1)}|\}$. Then, the sets $J^{(1)}_{\max}=J_{\ell_1}^{(1)}$ and $\SX_1=\CosSet\left(\Lambda(\boldz_{\ell_1}^{(1)},M_1)\right)$ can be used in order to continue the reduction of the index set $I_2=I_1\setminus J^{(1)}_{\max}$ under consideration.
We continue iteratively until we get $I_{K}=I_{K-1}\setminus J_{\max}^{(K-1)}=\emptyset$.

Obviously, the cardinalities of the mirrored index sets $\Mirror(I_j)$ as well as the expansions $N_{I_j}$, $j=1,\dots$, do not increase and consequently, the approach builds up a sequence of cosine transformed lattices $\SX_1,\SX_2,\dots$ with non-increasing lattice sizes $M_1,M_2,\dots$.

In the following, we estimate an upper bound on a specific failure probability.
To this end, we iteratively apply Theorem~\ref{thm:prob_bound_T_cheb} and we assume $I\subset\N_0^d$ and $\delta\in(0,1]$ are given and $L\ge \lceil 2(\log |I|-\log\delta)\rceil\in\N$ is determined and fixed.
We denote the remaining index sets under consideration by $I_j$, i.e., $I_0=I$, $I_1=I_0\setminus J^{(0)}_{\max}$,\dots, $I_{j+1}=I_j\setminus J^{(j)}_{\max}$,\dots.
The application of Theorem~\ref{thm:prob_bound_T_cheb} to a single step of the strategy explained above yields the following lemma.

\begin{lemma}\label{lem:prob_one_step_iterative}
We consider the frequency set $I_0\subset\N_0^d$, $|I_0|<\infty$, and $I_j\subset I_0$ which is a remaining index set arising from the strategy explained above.
In addition, $\delta\in(0,1]$ as well as $L\in\N$, $L\ge 2(\log|I_0|-\log\delta)$, are assumed to be fixed. Then, we choose a prime number $M_j$ such that $M_j>\max(2(|\Mirror(I_j)|-1), 2N_{I_j})$ and we draw $L$ generating vectors $\boldz_\ell\in[0,M_j-1]^d$, $\ell=1,\dots,L$, i.i.d. uniformly at random. Then the probability that none of the $L$ rank\mbox{-}1 lattices provides a Chebyshev matrix $\ChebMat(\CosSet(\Lambda(\boldz_\ell,M_j)),I_j)$ that has at least $|I_j|/L$ columns that fulfill the \NISOR condition is bounded from above by $\frac{|I_j|}{|I_0|}\delta$.
\end{lemma}
\begin{proof}
For fixed $\ell\in\{1,\dots,L\}$ we consider the Fourier matrix $\FourMat(\Lambda(\boldz_\ell,M_j),\Mirror(I_j))$ and we denote their columns by $\bolda_\boldh^{\ell}$, $\boldh\in\Mirror(I_j)$. In particular, we take a closer look at the columns $\bolda_\boldk^{\ell}$, $\boldk\in I_j$.
If we assume that there are at least $|I_j|/L$ columns within $\{\bolda_\boldk^{\ell}\colon \boldk\in I_j\}$, that fulfill the \NISOR condition w.r.t.\ $\FourMat(\Lambda(\boldz_\ell,M_j),\Mirror(I_j))$, then the corresponding Chebyshev matrix $\ChebMat(\CosSet(\Lambda(\boldz_\ell,M_j)),I_j)$ has at least $|I_j|/L$ columns that fulfill the \NISOR condition w.r.t.\ $\ChebMat(\CosSet(\Lambda(\boldz_\ell,M_j)),I_j)$ according to Section~\ref{sec:add_impro_general}. Consequently, we observe the following subset properties of several events, where we use the notation from the proofs of Theorems~\ref{thm:gen_mr1l_cheb} and~\ref{thm:prob_bound_T_cheb}.
\begingroup
\allowdisplaybreaks[3]
\begin{align*}
&\left\{\{\Lambda(\boldz_\ell,M_j)\colon \ell=1,\dots,L\}\colon \text{\parbox{8cm}{none of the $L$ rank\mbox{-}1 lattices $\Lambda(\boldz_\ell,M_j)$, provides a $\ChebMat(\CosSet(\Lambda(\boldz_\ell,M_j)),I_j)$ that has at least $|I_j|/L$ columns that fulfill the \NISOR condition}}\right\}\\
&\qquad \subset
\left\{\{\Lambda(\boldz_\ell,M_j)\colon \ell=1,\dots,L\}\colon \text{\parbox{8cm}{none of the $L$ rank\mbox{-}1 lattices $\Lambda(\boldz_\ell,M_j)$, provides a $\FourMat(\Lambda(\boldz_\ell,M_j),\Mirror(I_j))$ that has at least $|I_j|/L$ columns out of $\{\bolda_\boldk^{\ell}\colon \boldk\in I_j\}$ that fulfill the \NISOR condition}}\right\}\\
&\qquad = \left\{\min_{\ell\in\{1,\dots,L\}}\sum_{\boldk\in I_j}Y_\ell^\boldk > \frac{L-1}{L}|I_j|\right\}
=\left\{\max_{\ell\in\{1,\dots,L\}}\sum_{\boldk\in I_j}(1-Y_\ell^\boldk) < |I_j|/L\right\}\\
&\qquad \subset
\left\{\sum_{\ell=1}^L\sum_{\boldk\in I_j}(1-Y_\ell^\boldk) < |I_j|\right\}=\left\{\sum_{\boldk\in I_j}(L-\sum_{\ell=1}^L Y_\ell^\boldk) < |I_j|\right\}\\
&\qquad \subset \left\{\exists \boldk\in I_j\colon L-\sum_{\ell=1}^L Y_\ell^\boldk=0\right\}
=\bigcup_{\boldk\in I_j}\left\{\sum_{\ell=1}^LY_\ell^{\boldk}=L\right\}
\end{align*}
\endgroup
Clearly, the probability of the right hand side $\bigcup_{\boldk\in I_j}\left\{\sum_{\ell=1}^LY_\ell^{\boldk}=L\right\}$ is already estimated in Theorems~\ref{thm:prob_bound_T_cheb} and~\ref{thm:gen_mr1l_cheb}. Note that $c=2$ is fixed here. We obtain
$$
\mathcal{P}\left(\bigcup_{\boldk\in I_j}\left\{\sum_{\ell=1}^LY_\ell^{\boldk}=L\right\}\right)\le |I_j|\e^{-L/2}
\le 
|I_j|\e^{-\log|I_0|+\log\delta}=\frac{|I_j|}{|I_0|}\delta\,.
$$
\end{proof}
The last Lemma states that the probability that we cannot reduce $I_j$ by at least $|I_j|/L$ in the $(j+1)$th step is bounded by $\frac{|I_j|}{|I_0|}\delta$. We apply this to
\begin{align*}
&A:=\left\{\text{\parbox{8cm}{in any of the first $K$ steps none of the $L$ rank\mbox{-}1 lattices $\Lambda(\boldz_\ell,M_j)$, provides a $\ChebMat(\CosSet(\Lambda(\boldz_\ell,M_j)),I_j)$ that has at least $|I_j|/L$ columns that fulfill the \NISOR condition}}\right\}\\
&\qquad = \bigcup_{j=0}^{K-1} \left\{\text{\parbox{8cm}{the $(j+1)$th step is the first, where none of the $L$ rank\mbox{-}1 lattices $\Lambda(\boldz_\ell,M_j)$, provides a $\ChebMat(\CosSet(\Lambda(\boldz_\ell,M_j)),I_j)$ that has at least $|I_j|/L$ columns that fulfill the \NISOR condition}}\right\}
=: \bigcup_{j=0}^{K-1} A_j
\end{align*}
and exploit that in the event $A_j$ the relation $|I_{j}|\le \left(\frac{L-1}{L}\right)^{j}|I_0|$ holds. We gain
$$
\mathcal{P}(A)\le \sum_{j=0}^{K-1} \mathcal{P}(A_j)\le \frac{\delta}{|I_0|} \sum_{j=0}^{K-1}|I_j|  \le \delta \sum_{j=0}^{K-1} \left(\frac{L-1}{L}\right)^{j}\le L\delta
$$
for each $K\in\N$.

\begin{remark}\phantom{a}\\\label{rem:improvement_taking_sec21}
\begin{itemize}
\item 
For $r\ge 0$, we fix $L:=\max\left(10,2\left\lceil 2(1+r)\log |I|\right\rceil\right)$. Then, the above mentioned approach will determine a discretization of $\ChebPol(I)$ with probability at least $1-|I|^{-r}$. This is a consequence of the last considerations, when choosing $\delta=|I|^{-r}L^{-1}$ and taking $L-\lceil 2\log L\rceil\ge \frac{L}{2}$ for $L\ge 10$ into account. Note that $L=2\left\lceil 2(1+r)\log |I|\right\rceil$ holds for all $|I|\ge 8$ and $r\ge 0$.

\item Taking the aforementioned parameter choices into account, we estimate the number of steps which lead to a spatial discretization.
For $L\ge 2$, we compare the geometric series $\sum_{k=0}^\infty L^{-k}$ and the Taylor series of the exponential function, which yields

$$
\left(\frac{L}{L-1}\right)> \e^{1/L} \qquad\Rightarrow\qquad \left(\frac{L-1}{L}\right)^{\frac{L^2}{4(r+1)}}< \e^{-\frac{L}{4(r+1)}}\le \frac{1}{|I|}\,,
$$
i.e.,  $\left|I_{\frac{L^2}{4(r+1)}}\right| \le \left(\frac{L-1}{L}\right)^{\frac{L^2}{4(r+1)}}|I|<1$ in cases where no failure occurs. Accordingly, the proposed strategy yields a spatial discretization consisting of not more than $\frac{L^2}{4(r+1)}$ cosine transformed rank\mbox{-}1 lattices with probability $1-|I|^{-r}$ at least.

\item

For $|I|\ge 8$ with $N_I\le \frac{|\Mirror(I)|}{\log|I|}$, cf.~\eqref{eq:max_occur_powers}, we estimate the number of sampling nodes. With probability $1-|I|^{-r}$, we gain $|I_j|\le\left(\frac{L-1}{L}\right)^j|I|$ for all $j\in\N$ and, thus, we estimate the number of sampling nodes within $\SX=\CosSet\left(\bigcup_{j=1}^{K}\Lambda(\boldz_\ell,M_j)\right)$ by
\begin{align}
|\SX|&\le \sum_{j=0}^{K}
2\max\left(|\Mirror(I_j)|-1,N_{I_j}\right)\le 2\sum_{j=0}^{\frac{L^2}{4(r+1)}-1}\max\left(2^d\left(\frac{L-1}{L}\right)^j|I|,N_I\right)\nonumber\\
&\le 2\frac{|\Mirror(I)|}{\log|I|}\frac{L^2}{4(r+1)}+2^{d+1}|I|L\le 2^{d+1}|I|\left(\frac{L^2}{4(r+1)\log |I|}+L\right)\le \frac{9}{2}2^{d}L\,|I|\nonumber\\
&\le 45\,2^{d-1}(r+1)|I|\log|I|\in\OO{2^d|I|\ln|I|}\text{ for fixed $r$,}\label{eq:est_number_samples_532}
\end{align}
where the first inequality holds due to Remark~\ref{rem:remark_cards_cheb_lattices} taking Remark~\ref{rem:NISOR_at_least_two_columns} and that all columns of $\FourMat(\ST,\Mirror(I_j))$ are nonzero, i.e. $|\Mirror(I_{j+1})|\in\N_0\setminus\{1\}$, into account. The last inequalities in \eqref{eq:est_number_samples_532} are valid due to 
the general estimate $L\le 5(1+r)\log|I|$. Note that the upper bound on the number of sampling nodes $|\SX|$ has the same complexity as the one given in Section~\ref{sec:constr_wo_index_set} and the upper bound on the failure probability is also the same.
\item Limiting the number of iterations of the described approach to $\frac{L^2}{2(r+1)}$, we observe an algorithm that terminates after a finite number of steps. Clearly, the probability estimates still hold true for large enough $L$.
A rough estimate on the complexity of the algorithm leads to $\OO{|\Mirror(I)|(\log|I|+d)L^3}$, i.e., $\OO{|\Mirror(I)|(d+\log|I|)(\log|I|)^3}$ for $L\lesssim \log|I|$.
\end{itemize}
\end{remark}

\subsubsection{Further improvement}\label{sssec:further_improvement}

We additionally introduce some further improvement based on numerical observations.
After we have implemented the approach of the Section~\ref{sssec:improvement_simple} and considered the sets $J_\ell^{(j)}$, we find that their cardinality usually far exceeds the proportion $\frac{1}{L}$ of the cardinality of the index set $I_j$ under consideration. Often, many of the sets $J_\ell^{(j)}$ have cardinalities much greater than $|I_j|/2$, which is at least an indicator that possibly much smaller rank\mbox{-}1 lattice sizes can be used in the last mentioned approach (since therein we exploit that only a single $J_\ell^{(j)}$ necessarily has a cardinality $|I_j|/L$).

This leads directly to the obvious idea of reducing the lattice sizes in each step appropriately.
We describe one possibility that has proven successful in our numerical tests.
First we fix $L:=\max(10,2\left\lceil 2(1+r)\log |I|\right\rceil)$ and in the $(j+1)$th iteration, we determine 
$M_j=\operatorname{nextprime}(\max(c(|\Mirror(I_j)|-1),2N_{I_j}))$. Then we calculate the set $\mathcal{P}^{(j)}=\{p\in[3,M_j]\colon p\text{ prime}\}$ of all primes within $[3,M_j]$ and apply a \emph{sorted set halving method} similar to a bisection. In short words, we apply Algorithm~\ref{alg:determine_J_ell} to $I_j$, $L$, and the median(s) $P$ of $\mathcal{P}^{(j)}$. Subsequently, we consider the output and determine one set $J_\ell$ of highest cardinality and the associated $\boldz_\ell$. We refer to them as $|J_{\max}^{(j)}(P)|$ and $\boldz^{(j)}(P)$. We update the set $\mathcal{P}^{(j)}$
\begin{equation}
\mathcal{P}^{(j)}:=\begin{cases}
[P,\max(\mathcal{P}^{(j)})]\cap\mathcal{P}^{(j)},&\text{ $|J_{\max}^{(j)}(P)|< t^{(j)}$,} \\
[\min(\mathcal{P}^{(j)}),P]\cap\mathcal{P}^{(j)},&\text{ $|J_{\max}^{(j)}(P)|\ge t^{(j)}$,}
\end{cases}\label{eq:sorted_set_halving}
\end{equation}
where $t$ is an appropriate threshold.
Based on theory, at least $t^{(j)}\ge |I_j|/L$ should be selected.
We iteratively apply this sorted set halving method until the set $\mathcal{P}^{(j)}$ contains a last remaining element $P'$ and, thus, we acquire the corresponding rank\mbox{-}1 lattice $\Lambda(\boldz^{(j)}(P'),P')$. Depending on $t^{(j)}$, $P'$ might be significantly smaller than $M_j$.
We fix $\SX_j=\CosSet\left(\Lambda(\boldz^{(j)}(P'),P')\right)$, compute $I_{j+1}=I_j\setminus J_{\max}^{(j)}(P')$ and proceed further, i.e., we apply the algorithm explained above to $I_{j+1}$ with fixed $L:=\max(10,2\left\lceil 2(1+r)\log |I|\right\rceil)$.
Now we iteratively repeat the described procedure up to the point at which the remaining index set is empty. The union of the determined cosine transformed rank\mbox{-}1 lattices is a spatial discretization of $\ChebPol(I)$ due to Section~\ref{sec:JoiningSamplingSets}.
\begin{remark}\phantom{a}
\begin{itemize}
\item
Fixing the threshold $t^{(j)}\ge |I_j|/L$, the theoretical upper bounds on the failure probability and the theoretical upper bounds on the cardinality of the resulting sampling sets from Section~\ref{sssec:improvement_simple} still hold.
\item Assuming $N_I\lesssim |\Mirror(I)|$, we have $M_j\in\OO{|\Mirror(I)|}$
and, consequently, the additional sorting set halving method implemented in this section increases the computational effort by a factor of at most
$\log|\Mirror(I)|\lesssim d+\log|I|$ compared to the method described in Section~\ref{sssec:improvement_simple} which yields a total complexity in 
$$
\OO{|\Mirror(I)|(d+\log|I|)^2(\log|I|)^3}\subset \OO{|\Mirror(I)|d^2(\log|I|)^5},
$$
if the number of iterations is limited to $\frac{L^2}{2(r+1)}$, $L\lesssim\log|I|$,
cf.\ also Remark~\ref{rem:improvement_taking_sec21}.
\end{itemize}
\end{remark}

\section{Numerical tests}\label{sec:numtestcheb}

We implemented the algorithms suggested by Theorem~\ref{thm:prob_bound_T_cheb} and Sections~\ref{sssec:greedy_improvement}, \ref{sssec:improvement_simple}, \ref{sssec:further_improvement} in \uppercase{Matlab}$^{\text{\tiny \textregistered}}$. We fixed $r=1$ and $\delta=|I|^{-r}$, i.e., numbers $L=\lceil 4\log|I|\rceil$ when applying the algorithms proposed by Theorem~\ref{thm:prob_bound_T_cheb} and Section~\ref{sssec:greedy_improvement} and $L=\max\left(10,2\left\lceil 4\log |I|\right\rceil\right)$ when applying the algorithms developed in 
Sections~\ref{sssec:improvement_simple} and~\ref{sssec:further_improvement}.
According to the theory in Sections~\ref{sec:theory} and~\ref{sec:improvements}, these parameter choices guarantee failure probabilities not greater than $|I|^{-1}$.

In addition, the \emph{sorted set halving strategy} that we built into Section~\ref{sssec:further_improvement} requires certain threshold values $t^{(j)}\ge|I_j|/L$ which we have fixed with $t^{(j)}:=|I_j|/2$, cf.~\eqref{eq:sorted_set_halving}. Obviously, these threshold values are much higher than the theory suggests. 
However, several numerical tests, which we do not document here, have shown that when lower threshold values are taken into account, the sizes of the rank\mbox{-}1 lattices determined in the iteration steps do not become significantly smaller, but the number of rank\mbox{-}1 lattices used increases, which means that the total number of sampling nodes is generally higher.
In addition, even larger threshold values lead to significantly larger lattice sizes in the first steps, which often also lead to higher total numbers of sampling nodes.

Each numerical test was repeated ten times and the presented numbers of sampling nodes as well as the plotted condition numbers are the maximal values of the ten tests. Generally, the minimal values are comparable to the maximal values at least in its order of magnitude. In most reasonable cases, i.e., $|I|>100$ and when applying the algorithms developed in Sections~\ref{sssec:improvement_simple} and~\ref{sssec:further_improvement}, the minimal numbers of sampling nodes are greater than 90\% of the specified maximal numbers.

\subsection{Comparison to single rank\mbox{-}1 Chebyshev lattices}

In \cite[Tables~I and~II]{PoVo15}, the authors presented results for single rank\mbox{-}1 Chebyshev lattices that allow for the exact reconstruction of polynomials in the spans $\ChebPol(I_n^d)$ and $\ChebPol(H_n^d)$ of Chebyshev polynomials, where $I_n^d:=\{\boldk\in\N_0^d\colon\|\boldk\|_1\le n\}$ are $\ell_1$-ball index sets and $H_n^d:=\{\boldk\in\N_0^d\colon \prod_{j=1}^d\max(1,k_j)\le n\}$ are hyperbolic cross index sets. In fact, for a given index set $I\subset\N_0^d$ the specified algorithm computes a single rank\mbox{-}1 lattice $\Lambda(\boldz,2M)\subset\T^d$ such that the Chebyshev evaluation matrix $\ChebMat(\CosSet(\Lambda(\boldz,2M)),I)$ has full column rank. The sampling set $\CosSet(\Lambda(\boldz,2M))\subset[-1,1]^d$ has a cardinality of at most $M+1$.
We would like to point out, that finding relatively small $M$ is one crucial challenge and causes high computational effort which is roughly the same when applying the algorithms given by \cite{KuoMiNoNu19}. Accordingly, the rank\mbox{-}1 lattices that results in the given size parameters $M$ in \cite{PoVo15} are the outcome of resource-consuming computations.

In contrast to that, the newly developed strategies for determining the discretizations consisting of multiple cosine transformed rank\mbox{-}1 lattices are inexpensive at least in its complexity. Therefore, these strategies are better suited for use in adaptive algorithms. %

We have applied the newly developed algorithms to the index sets that has been numerically treated in \cite[Tables~I and~II]{PoVo15}. Tables~\ref{tab:cheb_l1ball} and~\ref{tab:cheb_symhc} present the resulting numbers of sampling nodes in columns four to seven. The last columns of the tables shows the results from \cite[Tables~I and~II]{PoVo15}.
\begin{table}
\footnotesize
\begin{tabular}{|r|r||r||r|r|r|r||r|}
\toprule
\multicolumn{2}{|c||}{\footnotesize \textsc{Parameters}}& \multicolumn{6}{c|}{\footnotesize \textsc{Cardinalities}}\\
\midrule
$d$ & $n$ & $|I_n^d|$ & {\scriptsize [Thm.~\ref{thm:prob_bound_T_cheb}]} $M_1$ & {\scriptsize [Sec.~\ref{sssec:greedy_improvement}]} $M_2$ & {\scriptsize [Sec.~\ref{sssec:improvement_simple}]} $M_3$ & {\scriptsize [Sec.~\ref{sssec:further_improvement}]} $M_4$ & {\scriptsize \cite{PoVo15}} $M$ \\
\midrule
2 & 64 & 2\,145 & 258\,045 & 8\,325 & 8\,325 & 5\,180 & 4\,192\\
2 & 128 & 8\,385 & 1\,222\,222 & 33\,034 & 33\,034 & 20\,743 & 16\,576\\
2 & 256 & 33\,153 & 5\,526\,571 & 131\,586 & 131\,586 & 82\,674 & 65\,920\\
\midrule
3 & 16 & 969 & 168\,505 & 12\,037 & 6\,019 & 3\,221 & 4\,265\\
3 & 32 & 6\,545 & 1\,650\,097 & 45\,837 & 45\,837 & 24\,037 & 33\,361\\
3 & 64 & 47\,905 & 15\,747\,073 & 715\,777 & 357\,889 & 188\,929 & 264\,353\\
\midrule
6 & 4 & 210 & 28\,359 & 3\,868 & 1\,304 & 667 & 1\,461\\
6 & 8 & 3\,003 & 1\,322\,740 & 120\,250 & 40\,238 & 16\,577 & 63\,369\\
6 & 16 & 74\,613 & 84\,823\,831 & 5\,654\,923 & 1\,885\,877 & 660\,292 & 3\,242\,322\\
\midrule
7 & 4 & 330 & 53\,761 & 6\,721 & 2\,264 & 1\,142 & 2\,777\\
7 & 8 & 6\,435 & 3\,907\,981 & 325\,666 & 108\,788 & 41\,871 & 223\,332\\
7 & 16 & 245\,157 & 458\,675\,801 & 27\,520\,549 & 9\,176\,575 & 2\,821\,570 & 21\,254\,517\\
\midrule
10 & 2 & 66 & 3758 & 443 & 242 & 144 & 202\\
10 & 4 & 1\,001 & 234\,193 & 25\,093 & 8\,489 & 4\,052 & 19\,423\\
10 & 8 & 43\,758 & 54\,028\,254 & 3\,769\,414 & 1\,258\,864 & 394\,630 & 5\,912\,807\\
\bottomrule
\end{tabular}
\caption{$\ell_1$-ball index sets and cardinalities of spatial discretizations.}\label{tab:cheb_l1ball}
\end{table}
\begin{table}
\footnotesize
\begin{tabular}{|r|r||r||r|r|r|r||r|}
\toprule
\multicolumn{2}{|c||}{\footnotesize \textsc{Parameters}}& \multicolumn{6}{c|}{\footnotesize \textsc{Cardinalities}}\\
\midrule
$d$ & $n$ & $|H_n^d|$ & {\scriptsize [Thm.~\ref{thm:prob_bound_T_cheb}]} $M_1$ & {\scriptsize [Sec.~\ref{sssec:greedy_improvement}]} $M_2$ & {\scriptsize [Sec.~\ref{sssec:improvement_simple}]} $M_3$ & {\scriptsize [Sec.~\ref{sssec:further_improvement}]} $M_4$ & {\scriptsize \cite{PoVo15}} $M$ \\
\midrule
2 & 256 & 1\,979 & 213\,591 & 20\,671 & 8\,182 & 5\,501 & 66\,050\\
2 & 512 & 4\,305 & 515\,781 & 60\,681 & 18\,695 & 12\,689 & 263\,170 \\
2 & 1\,024 & 9\,311 & 1\,226\,403 & 165\,731 & 41\,410 & 38\,041 & 1\,050\,626 \\
\midrule
3 & 256 & 10\,303 & 2\,228\,733 & 301\,181 & 66\,460 & 38\,977 & 302\,883 \\
3 & 512 & 23\,976 & 5\,872\,267 & 716\,131 & 160\,133 & 95\,091 & 1\,424\,613 \\
3 & 1\,024 & 55\,202 & 14\,785\,585 & 2\,016\,217 & 381\,582 & 227\,288 & 4\,600\,672 \\
\midrule
6 & 16 & 8\,684 & 6\,260\,808 & 507\,634 & 170\,099 & 63\,259 & 303\,396 \\
6 & 32 & 26\,088 & 22\,445\,861 & 2\,189\,841 & 552\,685 & 204\,472 & 1\,751\,513 \\
6 & 64 & 76\,433 & 76\,944\,961 & 6\,839\,553 & 1\,731\,541 & 638\,331 & 8\,979\,932 \\
\midrule
7 & 8 & 7\,184 & 7\,142\,977 & 595\,249 & 198\,668 & 63\,337 & 291\,267 \\
7 & 16 & 23\,816 & 29\,396\,755 & 2\,150\,983 & 720\,667 & 231\,583 & 1\,659\,143 \\
7 & 32 & 75\,532 & 110\,952\,676 & 9\,862\,461 & 2\,477\,496 & 784\,926 & 10\,375\,340 \\
\midrule
10 & 2 & 6\,144 & 15\,845\,341 & 905\,449 & 452\,725 & 97\,931 & 495\,451\\
10 & 4 & 27\,904 & 99\,261\,616 & 4\,842\,031 & 2\,421\,270 & 500\,766 & 3\,083\,988 \\
10 & 8 & 109\,824 & 508\,497\,748 & 32\,457\,304 & 10\,821\,189 & 2\,205\,847 & 25\,099\,619 \\
\bottomrule
\end{tabular}
\caption{Hyperbolic cross index sets and cardinalities of spatial discretizations.}
\label{tab:cheb_symhc}
\end{table}
Not surprisingly, the theoretical approach in Theorem~\ref{thm:prob_bound_T_cheb} which only depends on the cardinalities of the index sets $I$ and $\Mirror(I)$ does not yield spatial discretizations of comparable cardinality. This seems to be due to the fact that the number of necessary rank\mbox{-}1 lattices is clearly overestimated in our theory. However, even the greedy strategy in Section~\ref{sssec:greedy_improvement} leads to spatial discretizations that are almost comparable in their cardinality, i.e., we observe less than three times the cardinalities determined in \cite[Tables~I and~II]{PoVo15}. In some cases we are below the comparative figures.

The two more sophisticated strategies from 
Sections~\ref{sssec:improvement_simple} and~\ref{sssec:further_improvement} lead to even lower cardinalities of the spatial discretizations. In particular the strategy from Section~\ref{sssec:further_improvement} yields cardinalities which are in most cases much lower than those presented in 
\cite[Tables~I and~II]{PoVo15}. Obviously, in the case of $\ell_1$-ball index sets $I_n^2$ in two dimensions, we do not beat the single rank\mbox{-}1 Chebyshev lattice approach. In general, $\ell_1$-balls in two dimensions seem to fit very well to single rank\mbox{-}1 lattice approaches, which is at least indicated by the small oversampling factors $M/|I_n^2|$ less than two.
However, in all other cases we observe lower cardinalities of the spatial discretizations constructed using the strategy presented in Section~\ref{sssec:further_improvement}. In several cases, there are significant reductions in the number of sampling nodes down to factors of less than $1/20$.

\subsection{Dyadic hyperbolic crosses}

In this section we would like to focus on the difference between the construction of spatial discretizations for $\Pi(\Mirror(I))$ and the construction of spatial discretizations for $\ChebPol(I)$ using rank\mbox{-}1 lattices and their cosine transformed sampling nodes.
In particular, spatial discretizations of $\Pi(\Mirror(I))$ can be cosine transformed to spatial discretizations of $\ChebPol(I)$. Note that the  number of the lattice nodes are approximately halved after cosine transform.
In that sense, we can halve the number of nodes within spatial discretizations $\Pi(\Mirror(I))$ consisiting of rank\mbox{-}1 lattices and compare that number to the number of sampling nodes the newly developed approaches yield. For comparison, we use the results from \cite{Kae17} for dyadic hyperbolic cross index sets. To this end, we define the dyadic hyperbolic cross
$$
\tilde{H}_n^d:=\bigcap_{\|\boldj\|_1=n}G_{j_1}\times\dots\times G_{j_d},\qquad G_j=(2^{j-1},2^{j-1}]\cap\Z,
$$
which were used in \cite{Kae17}. A corresponding well-fitting index set for spans of Chebyshev polynomials is given by $\bar{H}_n^d:=\abs(\tilde{H}_n^d)$, which also has a dyadic construction and is downward closed, i.e.,
$$
\bar{H}_n^d=\bigcap_{\|\boldj\|_1=n}\bar{G}_{j_1}\times\dots\times \bar{G}_{j_d},\qquad \bar{G}_j=[0,2^{j-1}]\cap\Z.
$$
Due to the non-symmetry of $\tilde{H}_n^d$, we observe $\tilde{H}_n^d\subset\Mirror(\bar{H}_n^d)$. Accordingly, spatial discretizations of 
$\Pi(\tilde{H}_n^d)$ might not be spatial discretizations of $\Pi(\Mirror(\bar{H}_n^d))$ and a direct comparison is not possible. However, we will compare oversampling factors, i.e., the ratio of the number of sampling nodes and the dimension of the function spaces that are discretized. Based on theoretical considerations, we should expect oversampling factors to behave as $C\frac{|\Mirror(I)|\log|I|}{|I|}\lesssim 2^d\log|I|$. In particular, the order of magnitude of the oversampling factor is expected to be greater by a factor of up to $2^{d}$ compared to those observed for multiple rank\mbox{-}1 lattice spatial discretizations in the trigonometric case. Clearly, this factor is relaxed in cases where the number of simultaneously active variables is bounded below $d$, cf.\ Remark~\ref{rem:ds_dependence}. At this point, we should stress on the fact, that $\tilde{H}_n^d$ as well as $\bar{H}_n^d$ consists of integer vectors that have at most $\min(d,n)$ nonzero entries. Accordingly, the oversampling factors should be bounded from above by $C\,2^{\min(d,n)}\log|I|$, where $C$ is a constant independent of $d$ and $|I|$. Figure~\ref{fig:numtests_hypcross_nfix_cheb} shows oversampling factors for fixed $n\in\{2,3,4,5\}$ and growing $d$. Increasing $n$ by one, we should expect a factor of up to $2$ in the oversampling factors, which is roughly affirmed by the plots even for the less predictable approaches from Section~\ref{sec:constr_w_index_set}.
\begin{figure}[tb]%
\centering%
\begin{tikzpicture}[baseline=(current axis.south)]%
\begin{semilogxaxis}[
     clip=false,
     scale only axis,
     xmin=5,xmax=1500, ymin=0, ymax=150,
     xtick={10,28,78,210,561,1485},
     x tick label style={align=center, font=\scriptsize},
    ytick={1,8,16,32,64,128},
     y tick label style={font=\scriptsize},
     xticklabels={{10\\[0.3em]3},{28\\[0.3em]6},{78\\[0.3em]11},{210\\[0.3em]19},{561\\[0.3em]32},{1\,485\\[0.3em]53}},
    xlabel style={align=center, text width=7cm, font=\footnotesize},
    every axis legend/.append style={nodes={right}},
legend entries={Theorem~\ref{thm:prob_bound_T_cheb}, Section~\ref{sssec:greedy_improvement},Section~\ref{sssec:improvement_simple}, Section~\ref{sssec:further_improvement}
    },
    transpose legend,
    legend to name = leghypcrossnfix,
    legend columns=2,
    legend style={font=\footnotesize, /tikz/every even column/.append style={column sep=0.5cm}},
    title style={font=\footnotesize},
    title={$n=2$},
    width=0.4\textwidth,
    cycle list name=MR1LOF,
    ]
  \addplot coordinates {
(6, 1.883e+01)(10, 2.610e+01)(15, 3.013e+01)(21, 3.905e+01)(28, 4.304e+01)(36, 4.711e+01)(45, 5.193e+01)(55, 5.658e+01)(66, 5.694e+01)(78, 6.232e+01)(91, 6.578e+01)(105, 6.624e+01)(120, 7.101e+01)(136, 7.104e+01)(153, 7.481e+01)(171, 7.541e+01)(190, 7.583e+01)(210, 7.973e+01)(231, 8.058e+01)(253, 8.455e+01)(276, 8.442e+01)(300, 8.480e+01)(325, 8.899e+01)(351, 8.917e+01)(378, 8.946e+01)(406, 9.348e+01)(435, 9.339e+01)(465, 9.382e+01)(496, 9.390e+01)(528, 9.819e+01)(561, 9.798e+01)(595, 9.815e+01)(630, 9.868e+01)(666, 1.024e+02)(703, 1.024e+02)(741, 1.027e+02)(780, 1.028e+02)(820, 1.028e+02)(861, 1.067e+02)(903, 1.069e+02)(946, 1.070e+02)(990, 1.071e+02)(1035, 1.072e+02)(1081, 1.073e+02)(1128, 1.113e+02)(1176, 1.113e+02)(1225, 1.114e+02)(1275, 1.115e+02)(1326, 1.116e+02)(1378, 1.117e+02)(1431, 1.156e+02)(1485, 1.158e+02)
  };
  \addplot coordinates {
(6, 2.500e+00)(10, 2.700e+00)(15, 5.533e+00)(21, 6.048e+00)(28, 6.179e+00)(36, 6.306e+00)(45, 9.756e+00)(55, 6.673e+00)(66, 1.006e+01)(78, 1.040e+01)(91, 1.040e+01)(105, 1.047e+01)(120, 1.066e+01)(136, 1.066e+01)(153, 1.069e+01)(171, 1.078e+01)(190, 1.084e+01)(210, 1.450e+01)(231, 1.099e+01)(253, 1.103e+01)(276, 1.101e+01)(300, 1.106e+01)(325, 1.483e+01)(351, 1.115e+01)(378, 1.491e+01)(406, 1.496e+01)(435, 1.494e+01)(465, 1.501e+01)(496, 1.127e+01)(528, 1.511e+01)(561, 1.507e+01)(595, 1.510e+01)(630, 1.518e+01)(666, 1.517e+01)(703, 1.517e+01)(741, 1.522e+01)(780, 1.523e+01)(820, 1.524e+01)(861, 1.524e+01)(903, 1.528e+01)(946, 1.528e+01)(990, 1.530e+01)(1035, 1.532e+01)(1081, 1.533e+01)(1128, 1.536e+01)(1176, 1.535e+01)(1225, 1.537e+01)(1275, 1.538e+01)(1326, 1.540e+01)(1378, 1.540e+01)(1431, 1.541e+01)(1485, 1.544e+01)
 };
  \addplot coordinates {
(6, 2.500e+00)(10, 2.700e+00)(15, 2.933e+00)(21, 3.143e+00)(28, 3.214e+00)(36, 3.306e+00)(45, 3.467e+00)(55, 3.600e+00)(66, 3.667e+00)(78, 3.769e+00)(91, 3.758e+00)(105, 3.810e+00)(120, 3.883e+00)(136, 3.912e+00)(153, 3.915e+00)(171, 4.047e+00)(190, 4.079e+00)(210, 4.038e+00)(231, 4.117e+00)(253, 4.087e+00)(276, 4.116e+00)(300, 4.170e+00)(325, 4.172e+00)(351, 4.185e+00)(378, 4.225e+00)(406, 4.222e+00)(435, 4.248e+00)(465, 4.254e+00)(496, 4.250e+00)(528, 4.259e+00)(561, 4.287e+00)(595, 4.287e+00)(630, 4.319e+00)(666, 4.315e+00)(703, 4.337e+00)(741, 4.337e+00)(780, 4.354e+00)(820, 4.332e+00)(861, 4.340e+00)(903, 4.361e+00)(946, 4.410e+00)(990, 4.387e+00)(1035, 4.423e+00)(1081, 4.415e+00)(1128, 4.410e+00)(1176, 4.412e+00)(1225, 4.420e+00)(1275, 4.446e+00)(1326, 4.418e+00)(1378, 4.446e+00)(1431, 4.461e+00)(1485, 4.452e+00)
  };
  \addplot coordinates {
(6, 1.333e+00)(10, 2.000e+00)(15, 2.000e+00)(21, 1.905e+00)(28, 1.964e+00)(36, 2.000e+00)(45, 2.156e+00)(55, 2.145e+00)(66, 2.121e+00)(78, 2.179e+00)(91, 2.264e+00)(105, 2.276e+00)(120, 2.317e+00)(136, 2.353e+00)(153, 2.405e+00)(171, 2.497e+00)(190, 2.474e+00)(210, 2.495e+00)(231, 2.494e+00)(253, 2.549e+00)(276, 2.547e+00)(300, 2.600e+00)(325, 2.615e+00)(351, 2.604e+00)(378, 2.659e+00)(406, 2.655e+00)(435, 2.692e+00)(465, 2.729e+00)(496, 2.724e+00)(528, 2.748e+00)(561, 2.729e+00)(595, 2.773e+00)(630, 2.751e+00)(666, 2.766e+00)(703, 2.747e+00)(741, 2.789e+00)(780, 2.805e+00)(820, 2.815e+00)(861, 2.816e+00)(903, 2.834e+00)(946, 2.845e+00)(990, 2.856e+00)(1035, 2.843e+00)(1081, 2.858e+00)(1128, 2.854e+00)(1176, 2.856e+00)(1225, 2.889e+00)(1275, 2.900e+00)(1326, 2.877e+00)(1378, 2.894e+00)(1431, 2.908e+00)(1485, 2.908e+00)
  };
\node  at (axis description cs:0,0) [
    text width=31pt, anchor=north east, align=center,xshift=20pt,yshift=1.4pt] {\scriptsize $|\bar{H}_{2}^d|$};
\node  at (axis description cs:0,0) [
    text width=31pt, anchor=north east, align=center,xshift=20pt,yshift=-11.5pt] {\scriptsize $d$};
 \end{semilogxaxis}%
\end{tikzpicture}%
\hfill
\begin{tikzpicture}[baseline=(current axis.south)]%
  \begin{semilogxaxis}[
    clip=false,
    scale only axis,
    xmin=11,xmax=28000, ymin=0, ymax=150,
    xtick={23,90,375,1559,6577,27773},
    x tick label style={align=center, font=\scriptsize},
    ytick={1,8,16,32,64,128},
    y tick label style={font=\scriptsize},
    xticklabels={{23\\[0.3em]3},{90\\[0.3em]6},{375\\[0.3em]11},{1\,559\\[0.3em]19},{6\,577\\[0.3em]32},{27\,773\\[0.3em]53}},
    xlabel style={align=center, font=\footnotesize},
    title style={font=\footnotesize},
    title={$n=3$},
    width=0.4\textwidth,
    cycle list name=MR1LOF
    ]
  \addplot coordinates {
(12, 2.425e+01)(23, 3.848e+01)(39, 5.310e+01)(61, 6.774e+01)(90, 7.861e+01)(127, 9.292e+01)(173, 1.029e+02)(229, 1.132e+02)(296, 1.229e+02)(375, 1.324e+02)(467, 1.418e+02)%
  };
  \addplot coordinates {
(12, 2.500e+00)(23, 5.957e+00)(39, 7.103e+00)(61, 7.984e+00)(90, 1.311e+01)(127, 1.394e+01)(173, 1.471e+01)(229, 2.058e+01)(296, 1.603e+01)(375, 1.655e+01)(467, 1.701e+01)(573, 1.743e+01)(694, 1.779e+01)(831, 1.816e+01)(985, 1.842e+01)(1157, 1.869e+01)(1348, 2.525e+01)(1559, 2.555e+01)(1791, 2.583e+01)(2045, 2.608e+01)(2322, 2.630e+01)(2623, 2.652e+01)(2949, 2.672e+01)(3301, 2.691e+01)(3680, 2.708e+01)(4087, 2.724e+01)(4523, 2.739e+01)(4989, 2.753e+01)(5486, 2.766e+01)(6015, 2.779e+01)(6577, 2.791e+01)(7173, 2.802e+01)(7804, 3.517e+01)(8471, 2.823e+01)(9175, 2.832e+01)(9917, 2.842e+01)(10698, 2.850e+01)(11519, 2.859e+01)(12381, 2.867e+01)(13285, 3.592e+01)(14232, 2.881e+01)(15223, 2.888e+01)(16259, 3.619e+01)(17341, 2.901e+01)(18470, 3.634e+01)(19647, 3.641e+01)(20873, 3.648e+01)(22149, 3.655e+01)(23476, 3.661e+01)(24855, 3.667e+01)(26287, 3.674e+01)(27773, 3.679e+01)
 };
  \addplot coordinates {
(12, 2.500e+00)(23, 3.217e+00)(39, 3.692e+00)(61, 4.131e+00)(90, 4.544e+00)(127, 4.819e+00)(173, 5.104e+00)(229, 5.424e+00)(296, 5.608e+00)(375, 5.731e+00)(467, 5.893e+00)(573, 6.059e+00)(694, 6.195e+00)(831, 6.295e+00)(985, 6.407e+00)(1157, 6.485e+00)(1348, 6.550e+00)(1559, 6.669e+00)(1791, 6.734e+00)(2045, 6.789e+00)(2322, 6.858e+00)(2623, 6.895e+00)(2949, 6.947e+00)(3301, 6.983e+00)(3680, 7.032e+00)(4087, 7.098e+00)(4523, 7.148e+00)(4989, 7.158e+00)(5486, 7.188e+00)(6015, 7.225e+00)(6577, 7.268e+00)(7173, 7.276e+00)(7804, 7.320e+00)(8471, 7.355e+00)(9175, 7.349e+00)(9917, 7.379e+00)(10698, 7.397e+00)(11519, 7.410e+00)(12381, 7.433e+00)(13285, 7.450e+00)(14232, 7.475e+00)(15223, 7.478e+00)(16259, 7.503e+00)(17341, 7.516e+00)(18470, 7.536e+00)(19647, 7.536e+00)(20873, 7.556e+00)(22149, 7.576e+00)(23476, 7.582e+00)(24855, 7.599e+00)(26287, 7.603e+00)(27773, 7.618e+00)
  };
  \addplot coordinates {
(12, 1.750e+00)(23, 2.217e+00)(39, 2.103e+00)(61, 2.213e+00)(90, 2.544e+00)(127, 2.622e+00)(173, 2.775e+00)(229, 2.926e+00)(296, 3.051e+00)(375, 3.112e+00)(467, 3.195e+00)(573, 3.305e+00)(694, 3.406e+00)(831, 3.473e+00)(985, 3.477e+00)(1157, 3.550e+00)(1348, 3.607e+00)(1559, 3.628e+00)(1791, 3.638e+00)(2045, 3.679e+00)(2322, 3.715e+00)(2623, 3.744e+00)(2949, 3.746e+00)(3301, 3.775e+00)(3680, 3.798e+00)(4087, 3.817e+00)(4523, 3.841e+00)(4989, 3.867e+00)(5486, 3.879e+00)(6015, 3.885e+00)(6577, 3.892e+00)(7173, 3.922e+00)(7804, 3.928e+00)(8471, 3.946e+00)(9175, 3.959e+00)(9917, 3.971e+00)(10698, 3.982e+00)(11519, 3.986e+00)(12381, 3.996e+00)(13285, 4.013e+00)(14232, 4.013e+00)(15223, 4.020e+00)(16259, 4.035e+00)(17341, 4.033e+00)(18470, 4.039e+00)(19647, 4.048e+00)(20873, 4.059e+00)(22149, 4.059e+00)(23476, 4.069e+00)(24855, 4.076e+00)(26287, 4.086e+00)(27773, 4.088e+00)
  };
\node  at (axis description cs:0,0) [
    text width=31pt, anchor=north east, align=center,xshift=20pt,yshift=1.4pt] {\scriptsize $|\bar{H}_{3}^d|$};
\node  at (axis description cs:0,0) [
    text width=31pt, anchor=north east, align=center,xshift=20pt,yshift=-11.5pt] {\scriptsize $d$};
 \end{semilogxaxis}
\end{tikzpicture}\phantom{\rule{4pt}{1em}}%
\\[1em]%
\begin{tikzpicture}[baseline=(current axis.south)]%
\begin{semilogxaxis}[
    clip=false,
    scale only axis,
    xmin=24,xmax=400000, ymin=0, ymax=150,
    xtick={53,264,1519,9273,60025,397978},
    x tick label style={align=center, font=\scriptsize},
    ytick={1,8,16,32,64,128},
    y tick label style={font=\scriptsize},
    xticklabels={{53\\[0.3em]3},{264\\[0.3em]6},{1\,519\\[0.3em]11},{9\,273\\[0.3em]19},{60\,025\\[0.3em]32},{397\,978\\[0.3em]53}},
    xlabel style={align=center, text width=7cm,
    font=\footnotesize},
    title style={font=\footnotesize},
    title={{$n=4$}},
    width=0.4\textwidth,
    cycle list name=MR1LOF
    ]
  \addplot coordinates {
(25, 3.384e+01)(53, 5.315e+01)(98, 7.834e+01)(166, 1.012e+02)(264, 1.270e+02)(400, 1.479e+02)%
  };
  \addplot coordinates {
(25, 5.240e+00)(53, 6.660e+00)(98, 1.238e+01)(166, 1.446e+01)(264, 1.657e+01)(400, 1.849e+01)(583, 2.026e+01)(823, 2.188e+01)(1131, 3.114e+01)(1519, 3.293e+01)(2000, 3.454e+01)(2588, 3.601e+01)(3298, 3.736e+01)(4146, 3.860e+01)(5149, 3.973e+01)(6325, 4.077e+01)(7693, 4.174e+01)(9273, 4.262e+01)(11086, 4.344e+01)(13154, 5.526e+01)(15500, 4.492e+01)(18148, 5.698e+01)(21123, 5.776e+01)(24451, 5.849e+01)(28159, 5.917e+01)(32275, 5.981e+01)(36828, 6.041e+01)(41848, 6.098e+01)(47366, 6.152e+01)(53414, 6.203e+01)(60025, 6.251e+01)(67233, 6.297e+01)(75073, 6.341e+01)(83581, 6.382e+01)(92794, 6.421e+01)(102750, 6.459e+01)(113488, 6.495e+01)(125048, 6.529e+01)(137471, 6.561e+01)(150799, 6.593e+01)(165075, 6.623e+01)(180343, 6.651e+01)(196648, 6.679e+01)(214036, 6.705e+01)(232554, 6.731e+01)(252250, 6.755e+01)(273173, 8.135e+01)(295373, 6.802e+01)(318901, 6.823e+01)(343809, 6.845e+01)(370150, 6.865e+01)(397978, 6.885e+01) 
 };
  \addplot coordinates {
(25, 2.640e+00)(53, 3.491e+00)(98, 4.245e+00)(166, 5.000e+00)(264, 5.674e+00)(400, 6.447e+00)(583, 6.962e+00)(823, 7.484e+00)(1131, 8.027e+00)(1519, 8.438e+00)(2000, 8.842e+00)(2588, 9.194e+00)(3298, 9.543e+00)(4146, 9.865e+00)(5149, 1.011e+01)(6325, 1.039e+01)(7693, 1.062e+01)(9273, 1.084e+01)(11086, 1.104e+01)(13154, 1.122e+01)(15500, 1.141e+01)(18148, 1.157e+01)(21123, 1.170e+01)(24451, 1.184e+01)(28159, 1.198e+01)(32275, 1.211e+01)(36828, 1.222e+01)(41848, 1.233e+01)(47366, 1.244e+01)(53414, 1.253e+01)(60025, 1.263e+01)(67233, 1.271e+01)(75073, 1.280e+01)(83581, 1.288e+01)(92794, 1.296e+01)(102750, 1.303e+01)(113488, 1.310e+01)(125048, 1.317e+01)(137471, 1.323e+01)(150799, 1.328e+01)(165075, 1.334e+01)(180343, 1.340e+01)(196648, 1.345e+01)(214036, 1.351e+01)(232554, 1.355e+01)(252250, 1.360e+01)(273173, 1.365e+01)(295373, 1.369e+01)(318901, 1.373e+01)(343809, 1.377e+01)(370150, 1.381e+01)(397978, 1.385e+01)
  };
  \addplot coordinates {
(25, 2.120e+00)(53, 2.208e+00)(98, 2.316e+00)(166, 2.777e+00)(264, 3.015e+00)(400, 3.303e+00)(583, 3.497e+00)(823, 3.746e+00)(1131, 3.989e+00)(1519, 4.145e+00)(2000, 4.191e+00)(2588, 4.349e+00)(3298, 4.439e+00)(4146, 4.561e+00)(5149, 4.672e+00)(6325, 4.737e+00)(7693, 4.776e+00)(9273, 4.850e+00)(11086, 4.919e+00)(13154, 4.991e+00)(15500, 5.037e+00)(18148, 5.075e+00)(21123, 5.130e+00)(24451, 5.162e+00)(28159, 5.216e+00)(32275, 5.248e+00)(36828, 5.294e+00)(41848, 5.305e+00)(47366, 5.339e+00)(53414, 5.375e+00)(60025, 5.401e+00)(67233, 5.409e+00)(75073, 5.435e+00)(83581, 5.458e+00)(92794, 5.483e+00)(102750, 5.513e+00)(113488, 5.529e+00)(125048, 5.546e+00)(137471, 5.563e+00)(150799, 5.583e+00)(165075, 5.596e+00)(180343, 5.616e+00)(196648, 5.624e+00)(214036, 5.645e+00)(232554, 5.655e+00)(252250, 5.667e+00)(273173, 5.684e+00)(295373, 5.695e+00)(318901, 5.716e+00)(343809, 5.718e+00)(370150, 5.732e+00)(397978, 5.744e+00)
  };
\node  at (axis description cs:0,0) [
    text width=31pt, anchor=north east, align=center,xshift=20pt,yshift=1.4pt] {\scriptsize $|\bar{H}_{4}^d|$};
\node  at (axis description cs:0,0) [
    text width=31pt, anchor=north east, align=center,xshift=20pt,yshift=-11.5pt] {\scriptsize $d$};
 \end{semilogxaxis}%
\end{tikzpicture}%
\hfill
\begin{tikzpicture}[baseline=(current axis.south)]%
\begin{semilogxaxis}[%
    clip=false,
    scale only axis,
    xmin=24,xmax=5000000,ymin=0,ymax=150,
    xtick={122, 738, 5534, 47330, 456089, 4666757},
    x tick label style={align=center, font=\scriptsize},
    ytick={1,8,16,32,64,128},
    y tick label style={font=\scriptsize},
	xticklabels={{122\\[0.3em]3},{738\\[0.3em]6},{5\,534\\[0.3em]11},{47\,330\\[0.3em]19},{456\,089\\[0.3em]32},{4\,666\,757\\[0.3em]53}},
    xlabel style={align=center,%
    font=\footnotesize},
    title style={font=\footnotesize},
    title={$n=5$},
    width=0.4\textwidth,
    cycle list name=MR1LOF
    ]
  \addplot coordinates {
  (53, 4.409e+01)(122, 7.214e+01)(242, 1.005e+02)(437, 1.393e+02)%
  };
  \addplot coordinates {
(53, 5.528e+00)(122, 1.083e+01)(242, 1.371e+01)(437, 1.672e+01)(738, 1.978e+01)(1184, 3.047e+01)(1823, 3.452e+01)(2713, 3.837e+01)(3923, 4.207e+01)(5534, 5.701e+01)(7640, 4.895e+01)(10349, 5.210e+01)(13784, 6.882e+01)(18084, 7.230e+01)(23405, 7.557e+01)(29921, 7.863e+01)(37825, 8.151e+01)(47330, 8.422e+01)(58670, 8.677e+01)(72101, 8.916e+01)(87902, 9.142e+01)(106376, 9.355e+01)(127851, 9.556e+01)(152681, 9.746e+01)(181247, 9.926e+01)(213958, 1.010e+02)(251252, 1.026e+02)(293597, 1.041e+02)(341492, 1.267e+02)(395468, 1.070e+02)(456089, 1.083e+02)(523953, 1.315e+02)(599693, 1.329e+02)(683978, 1.119e+02)(777514, 1.356e+02)(881045, 1.368e+02)(995354, 1.381e+02)(1121264, 1.392e+02)(1259639, 1.403e+02)(1411385, 1.414e+02)(1577451, 1.424e+02)(1758830, 1.434e+02)(1956560, 1.443e+02)(2171725, 1.452e+02)(2405456, 1.461e+02)(2658932, 1.470e+02)(2933381, 1.478e+02)(3230081, 1.486e+02)(3550361, 1.493e+02)%
};
  \addplot coordinates {
(53, 3.113e+00)(122, 3.779e+00)(242, 4.835e+00)(437, 5.847e+00)(738, 6.843e+00)(1184, 7.878e+00)(1823, 8.957e+00)(2713, 9.814e+00)(3923, 1.075e+01)(5534, 1.161e+01)(7640, 1.243e+01)(10349, 1.322e+01)(13784, 1.395e+01)(18084, 1.465e+01)(23405, 1.526e+01)(29921, 1.585e+01)(37825, 1.643e+01)(47330, 1.697e+01)(58670, 1.746e+01)(72101, 1.794e+01)(87902, 1.838e+01)(106376, 1.880e+01)(127851, 1.920e+01)(152681, 1.957e+01)(181247, 1.993e+01)(213958, 2.026e+01)(251252, 2.058e+01)(293597, 2.089e+01)(341492, 2.118e+01)(395468, 2.145e+01)(456089, 2.171e+01)(523953, 2.196e+01)(599693, 2.220e+01)(683978, 2.243e+01)(777514, 2.264e+01)(881045, 2.285e+01)(995354, 2.305e+01)(1121264, 2.324e+01)(1259639, 2.343e+01)(1411385, 2.360e+01)(1577451, 2.377e+01)(1758830, 2.393e+01)(1956560, 2.409e+01)(2171725, 2.424e+01)(2405456, 2.438e+01)(2658932, 2.452e+01)(2933381, 2.466e+01)(3230081, 2.479e+01)(3550361, 2.491e+01)(3895602, 2.504e+01)(4267238, 2.515e+01)(4666757, 2.527e+01)
  };
  \addplot coordinates {
(53, 2.113e+00)(122, 2.246e+00)(242, 2.682e+00)(437, 3.098e+00)(738, 3.526e+00)(1184, 3.959e+00)(1823, 4.196e+00)(2713, 4.551e+00)(3923, 4.951e+00)(5534, 5.004e+00)(7640, 5.308e+00)(10349, 5.439e+00)(13784, 5.655e+00)(18084, 5.823e+00)(23405, 5.916e+00)(29921, 6.036e+00)(37825, 6.186e+00)(47330, 6.308e+00)(58670, 6.437e+00)(72101, 6.554e+00)(87902, 6.652e+00)(106376, 6.835e+00)(127851, 6.802e+00)(152681, 6.897e+00)(181247, 6.973e+00)(213958, 7.034e+00)(251252, 7.112e+00)(293597, 7.185e+00)(341492, 7.258e+00)(395468, 7.314e+00)(456089, 7.381e+00)(523953, 7.435e+00)(599693, 7.493e+00)(683978, 7.542e+00)(777514, 7.586e+00)(881045, 7.636e+00)(995354, 7.690e+00)(1121264, 7.743e+00)(1259639, 7.776e+00)(1411385, 7.819e+00)(1577451, 7.858e+00)(1758830, 7.904e+00)(1956560, 7.930e+00)(2171725, 7.976e+00)(2405456, 8.010e+00)(2658932, 8.046e+00)(2933381, 8.081e+00)(3230081, 8.114e+00)(3550361, 8.142e+00)(3895602, 8.174e+00)(4267238, 8.201e+00)(4666757, 8.233e+00)
  };
\node  at (axis description cs:0,0) [
    text width=31pt, anchor=north east, align=center,xshift=20pt,yshift=1.4pt] {\scriptsize $|\bar H_{5}^d|$};
\node  at (axis description cs:0,0) [
    text width=31pt, anchor=north east, align=center,xshift=20pt,yshift=-11.5pt] {\scriptsize $d$};
 \end{semilogxaxis}
\end{tikzpicture}%
\\[1em]
\ref{leghypcrossnfix}
\caption{
Oversampling factors of spatial discretizations consisting of cosine transformed rank\mbox{-}1 lattices for $\ChebPol(\bar{H}_n^d)$, $n\in\{2,3,4,5\}$ fixed and dimensions $d$ up to $53$.
}\label{fig:numtests_hypcross_nfix_cheb}
\end{figure}
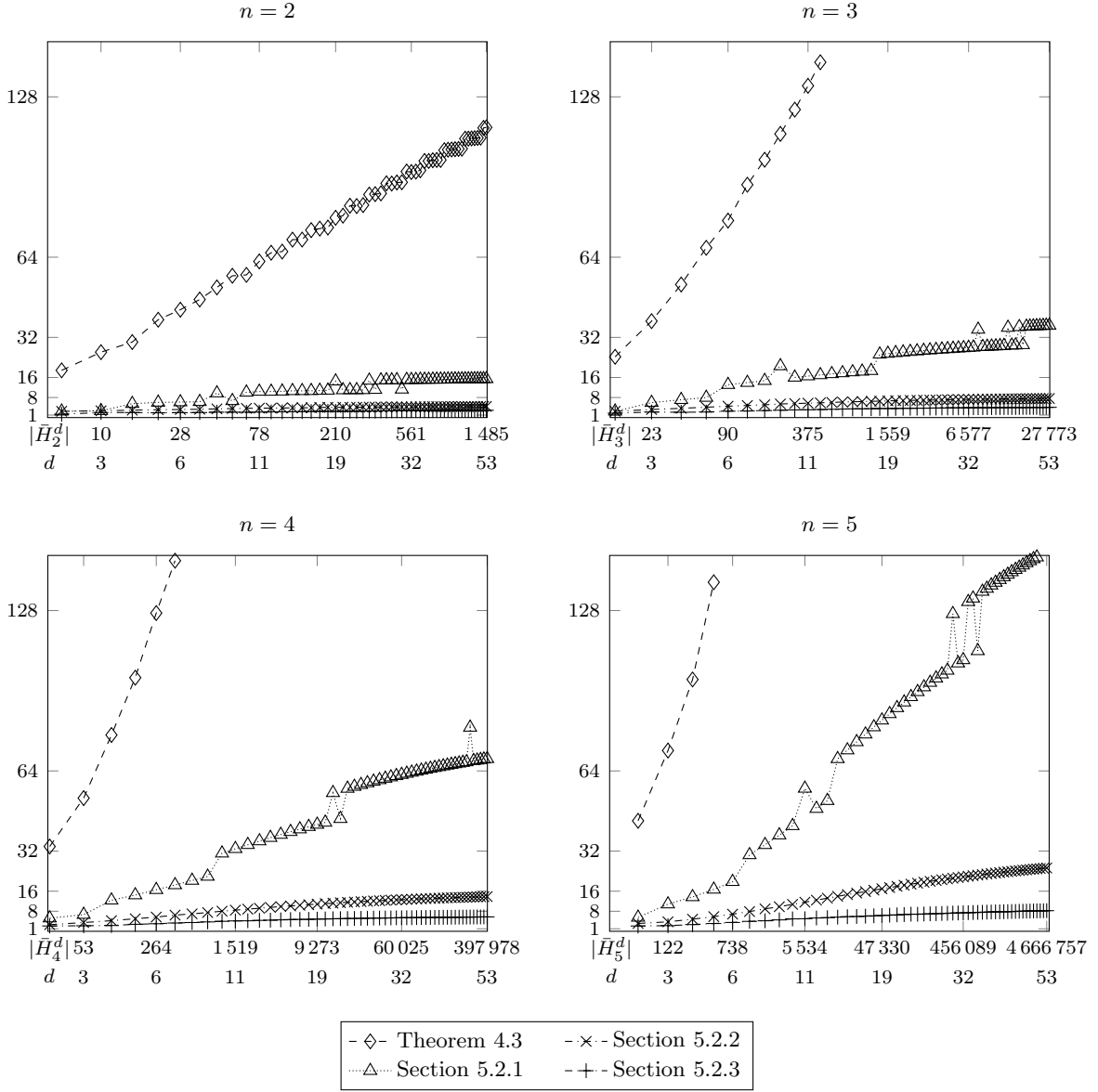

Similar to \cite[Fig. 5.2]{Kae17}, we observe mildly growing or even stagnating oversampling factors for fixed $n$ and growing $d$. In absolute numbers, the oversampling factors of the most sophisticated approach, cf.\ Section~\ref{sssec:further_improvement}, are less than nine and, thus, entirely acceptable.

Moreover, we applied the developed algorithms for constructing spatial discretizations of $\ChebPol(\bar{H}_n^6)$, $1 \le n\le 13$. Figure~\ref{fig:numtests_hypcross_cheb_d6} shows the oversampling factors in a diagram that has logarithmic scales on both axes.  In addition, we have added dotted lines to indicate the dependencies of the oversampling factors on the cardinality of the respective index sets. All four algorithms appear to provide discretizations that have oversampling factors that actually depend logarithmically on the cardinality of the index sets. 
However, the oversampling factors of the approach developed in section~\ref{sssec:further_improvement} are less than $6.1$ even for $\bar{H}_{13}^6$, which has a cardinality of over one million.

Further, we computed condition numbers of matrices $\ChebMat(\SX,\bar{H}_n^6)$, $1\le n\le 8$. Figure~\ref{fig:numtests_hypcross_cheb_d6} shows the maximum condition numbers of the respective ten tests.
The computed condition numbers are comparatively low. 
In fact, in all our numerical tests where we calculated condition numbers, we found low condition numbers of a similar order of magnitude, even when we considered other index sets.
The surprising observation here is that the approach of Section~\ref{sssec:further_improvement} leads to lower (maximal) condition numbers than the approach of Section~\ref{sssec:improvement_simple} for $H_n^6$, $4\le n\le8$, even though fewer sampling nodes are taken.
However, this effect is generally not observed.

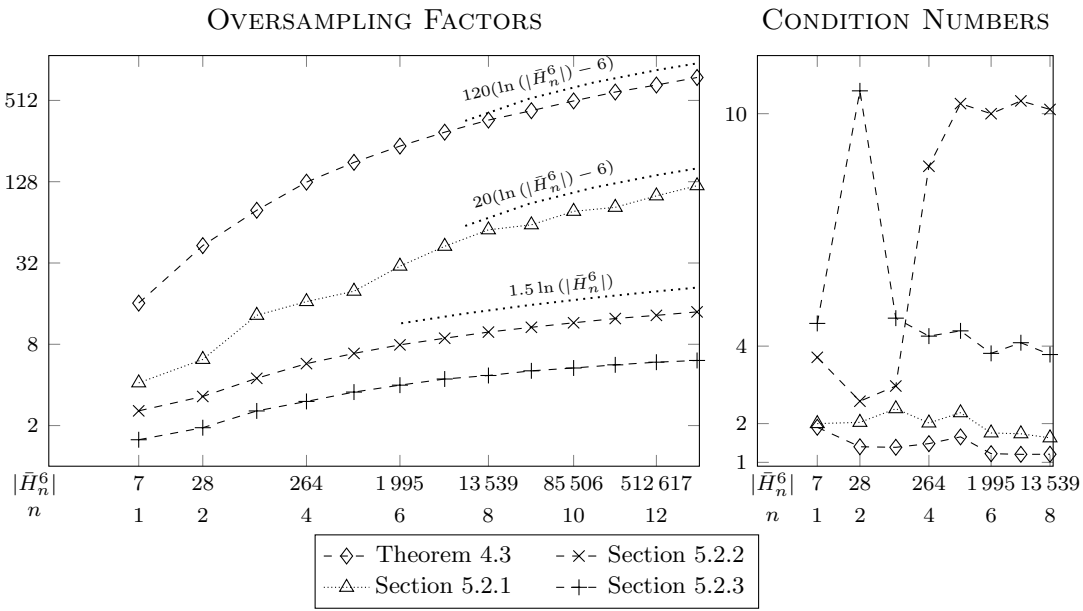
\begin{figure}[bt]
\centering
\begin{tikzpicture}[baseline=(current axis.south)]%
\begin{loglogaxis}[%
    clip=false,
    scale only axis,
    xmin=1,xmax=1300000, ymin=1, ymax=1100,
    title={\textsc{Oversampling Factors}},
    xtick={7,28,264,1995,13539,85506,512617},
    ytick={2,8,32,128,512},
    yticklabels={2,8,32,128,512},
    x tick label style={align=center,font=\scriptsize},
    y tick label style={font=\scriptsize},
    xticklabels={{7\\[0.3em]1},{28\\[0.3em]2},{264\\[0.3em]4},{1\,995\\[0.3em]6},{13\,539\\[0.3em]8}, {85\,506\\[0.3em]10},{512\,617\\[0.3em]12}},
    width=0.55\textwidth,
    height=0.35\textwidth,
legend entries={Theorem~\ref{thm:prob_bound_T_cheb}, Section~\ref{sssec:greedy_improvement},Section~\ref{sssec:improvement_simple}, Section~\ref{sssec:further_improvement}
    },
    transpose legend,
    legend to name = leghypcross6dcheb,
    legend columns=2,
    legend style={font=\footnotesize, /tikz/every even column/.append style={column sep=0.5cm}},
    cycle list name=MR1LOF
    ]
  \addplot coordinates {
(7, 1.614e+01)(28, 4.304e+01)(90, 7.861e+01)(264, 1.270e+02)(738, 1.780e+02)(1995, 2.349e+02)(5253, 2.978e+02)(13539, 3.661e+02)(34281, 4.296e+02)(85506, 5.076e+02)(210540, 5.905e+02)(512617, 6.655e+02)(1235821, 7.569e+02)
};
  \addplot coordinates {
(7, 4.143e+00)(28, 6.179e+00)(90, 1.311e+01)(264, 1.657e+01)(738, 1.978e+01)(1995, 3.032e+01)(5253, 4.254e+01)(13539, 5.633e+01)(34281, 6.137e+01)(85506, 7.724e+01)(210540, 8.267e+01)(512617, 1.005e+02)(1235821, 1.195e+02)
};
  \addplot+[dashed] coordinates {
(7, 2.571e+00)(28, 3.286e+00)(90, 4.478e+00)(264, 5.750e+00)(738, 6.848e+00)(1995, 7.922e+00)(5253, 8.871e+00)(13539, 9.851e+00)(34281, 1.069e+01)(85506, 1.154e+01)(210540, 1.241e+01)(512617, 1.310e+01)(1235821, 1.390e+01)
};
  \addplot+[dashed] coordinates {
(7, 1.571e+00)(28, 1.929e+00)(90, 2.567e+00)(264, 3.019e+00)(738, 3.541e+00)(1995, 3.985e+00)(5253, 4.425e+00)(13539, 4.698e+00)(34281, 5.088e+00)(85506, 5.339e+00)(210540, 5.635e+00)(512617, 5.895e+00)(1235821, 6.071e+00)
};
  \addplot[dotted, line width=0.75pt] coordinates {
 (8192,3.613e+02) (32768,5.277e+02) (131072,6.940e+02) (524288,8.604e+02) (1300000,9.693e+02)
} node[left, xshift=-28pt, yshift=1.5pt, rotate=13] {\tiny $120(\ln (|\bar H_n^6|)-6)$};
  \addplot[dotted, line width=0.75pt] coordinates {
 (8192,6.022e+01) (32768,8.794e+01) (131072,1.157e+02) (524288,1.434e+02) (1300000,1.616e+02)
} node[left, xshift=-28pt, yshift=2pt, rotate=15] {\tiny $20(\ln (|\bar H_n^6|)-6)$};
  \addplot[dotted, line width=0.75pt] coordinates {
(2048,1.144e+01) (8192,1.352e+01) (32768,1.560e+01) (131072,1.768e+01) (524288,1.975e+01) (1300000,2.112e+01)
} node[left, xshift=-28pt, yshift=3pt, rotate=6] {\tiny $1.5\,\ln (|\bar H_n^6|)$};
\node  at (axis description cs:0,0) [text width=31pt, anchor=north east, align=center,xshift=14pt,yshift=1.5pt] {\scriptsize $|\bar H_{n}^6|$};
\node  at (axis description cs:0,0) [text width=31pt, anchor=north east, align=center,xshift=14pt,yshift=-11.5pt] {\scriptsize $n$};
\end{loglogaxis}
\end{tikzpicture}
\begin{tikzpicture}[baseline=(current axis.south)]%
\begin{semilogxaxis}[%
    clip=false,
    scale only axis,
    xmin=1,xmax=15000, ymin=0.9, ymax=11.5,
    title={\textsc{Condition Numbers}},
    xtick={7,28,264,1995,13539},
    ytick={1,2,4,10},
    x tick label style={align=center,font=\scriptsize},
    y tick label style={font=\scriptsize},
    xticklabels={{7\\[0.3em]1},{28\\[0.3em]2},{264\\[0.3em]4},{1\,995\\[0.3em]6},{13\,539\\[0.3em]8}},
    yticklabels={1,2,4,10},
    width=0.25\textwidth,
    height=0.35\textwidth,
    cycle list name=MR1LOF
    ]
  \addplot coordinates {
(7, 1.905e+00)(28, 1.402e+00)(90, 1.388e+00)(264, 1.483e+00)(738, 1.657e+00)(1995, 1.224e+00)(5253, 1.206e+00)(13539, 1.209e+00)
};
  \addplot coordinates {
(7, 2.000e+00)(28, 2.032e+00)(90, 2.384e+00)(264, 2.016e+00)(738, 2.280e+00)(1995, 1.759e+00)(5253, 1.737e+00)(13539, 1.632e+00)
};
  \addplot+[dashed] coordinates {
(7, 3.711e+00)(28, 2.578e+00)(90, 2.970e+00)(264, 8.643e+00)(738, 1.026e+01)(1995, 9.999e+00)(5253, 1.033e+01)(13539, 1.011e+01)
};
  \addplot+[dashed] coordinates {
(7, 4.586e+00)(28, 1.059e+01)(90, 4.722e+00)(264, 4.258e+00)(738, 4.394e+00)(1995, 3.813e+00)(5253, 4.088e+00)(13539, 3.783e+00)
};
\node  at (axis description cs:0,0) [text width=31pt, anchor=north east, align=center,xshift=25pt,yshift=1.4pt] {\scriptsize $|\bar{H}_{n}^6|$};
\node  at (axis description cs:0,0) [text width=31pt, anchor=north east, align=center,xshift=25pt,yshift=-13.5pt] {\scriptsize $n$};
\end{semilogxaxis}
\end{tikzpicture}
\ref{leghypcross6dcheb}
\caption{Oversampling factors of spatial discretizations consisting of cosine transformed rank\mbox{-}1 lattices for spans of Chebyshev polynomials with dyadic hyperbolic cross index sets $\bar H_n^6$ and the condition numbers of associated Chebyshev matrices.}
\label{fig:numtests_hypcross_cheb_d6}
\end{figure}

\subsection{Random index sets with fixed $d_s$}

According to Remark~\ref{rem:ds_dependence}, the dependencies of the cardinalities of the spatial discretizations mainly depend on the (maximal) number of interacting variables, i.e., on the (maximal) number of different variables $d_s<d$ on which the basis polynomials $T_\boldk$ depend. For a single $T_\boldk$ this number is given by $\|\boldk\|_0$.
In this numerical test, we will draw random index sets $I\subset[0,1024]^{25}$, where $\|\boldk\|_0=d_s$ is fixed for each $\boldk\in I$ in order to inspect the actual dependencies on $d_s$.
On the one hand, for fixed cardinality $|I|$ the larger $d_s$ is, the larger the set $\Mirror(I)$ is, i.e., $|\Mirror(I)|=2^{d_s}|I|$, which generally leads to a larger number of possible aliasing frequencies within $\Mirror(I)$ in the trigonometric system, cf.\ Lemma~\ref{lem:count_aliasing_probability_cheb}. On the other hand, for each $\boldk\in I$, we need only one column $a_\boldh$, $\boldh\in\Mirror(\{\boldk\})$, of $\FourMat(\ST,I)$, which fulfills the \NISOR condition, cf.\ Corollary~\ref{cor:matrix_C_full_crank}, i.e., which might decrease the necessary number of samples compared to the theoretical results in Theorem~\ref{thm:prob_bound_T_cheb} since this theoretical result guarantees that each specific $a_\boldk$ -- and not one out of  $\{a_\boldh \colon \boldh\in\Mirror(\{\boldk\})\}$ -- fulfills the \NISOR property. Clearly, for growing $d_s$ the cardinalities of the sets $\Mirror(\{\boldk\})$ increase and, thus, a larger number of columns $a_\boldh$ are available, of which only one should fulfill the \NISOR condition.

We show oversampling factors $|\SX|/|I|$ in Figure~\ref{fig:numtests_rand_cheb_d25_ds}
for different $d_s\in\{2,3,4,5,6,7,8,9\}$ and different $|I|\in\{2^k\colon k\in\N, 4\le k\le 16\}$. Here, we used the most sophisticated approach from Section~\ref{sssec:further_improvement}.
For fixed $d_s$, we observe (almost) stagnating oversampling factors.
As expected the oversampling factors increase with increasing $d_s$. However, the increase appears to be somewhat milder compared to the theory, i.e., $|\SX|/|I|\gtrsim 2^{d_s}$. In fact, all obtained oversampling factors fulfill $\frac{|\SX|/|I|}{2^{d_s}/d_s}\in (1,2)$, i.e., we observe oversampling factors that are at least a factor $d_s/2$ better than expected by theory.

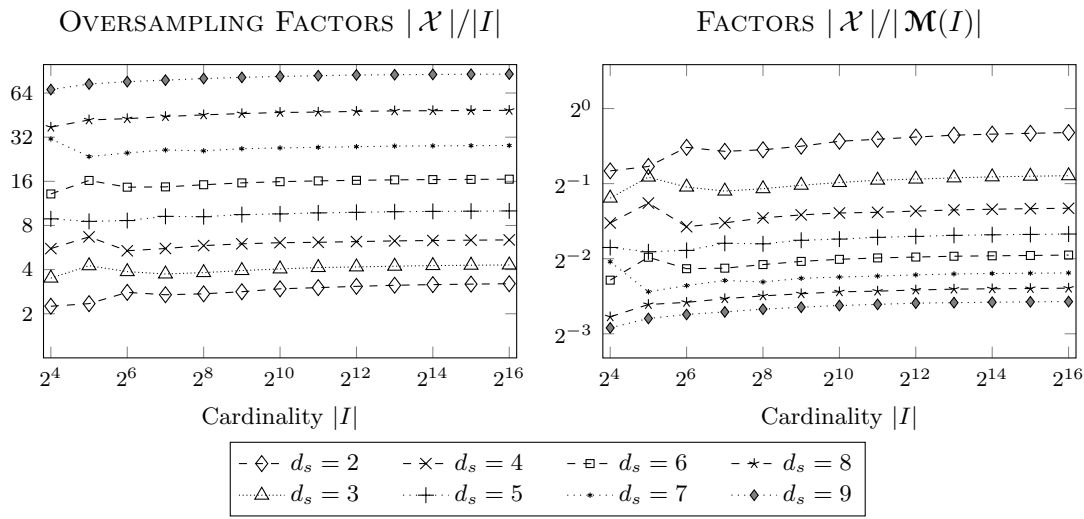
\begin{figure}
\centering
\begin{tikzpicture}[baseline=(current axis.south)]%
\begin{loglogaxis}[%
    clip=false,
    scale only axis,
    xmin=14,xmax=75000, ymin=1, ymax=100,
    title={\textsc{Oversampling Factors $|\SX|/|I|$}},
    xtick={16,64,256,1024,4096,16384,65536},
    ytick={2,4,8,16,32,64},
    yticklabels={2,4,8,16,32,64},
    x tick label style={align=center,font=\scriptsize},
    y tick label style={font=\scriptsize},
    xticklabels={$2^4$,$2^6$,$2^8$,$2^{10}$,$2^{12}$,$2^{14}$,$2^{16}$},
    x label style={font=\footnotesize},
    xlabel={Cardinality $|I|$},
    width=0.4\textwidth,
    height=0.25\textwidth,
legend entries={$d_s=2$, $d_s=3$, $d_s=4$, $d_s=5$, $d_s=6$, $d_s=7$, $d_s=8$, $d_s=9$},
    transpose legend,
    legend to name = legrandcheb,
    legend columns=2,
    legend style={font=\footnotesize, /tikz/every even column/.append style={column sep=0.5cm}},
    cycle list name=MR1LOF
    ]
  \addplot coordinates {
(16, 2.250e+00)(32, 2.344e+00)(64, 2.797e+00)(128, 2.695e+00)(256, 2.734e+00)(512, 2.826e+00)(1024, 2.958e+00)(2048, 3.014e+00)(4096, 3.075e+00)(8192, 3.128e+00)(16384, 3.160e+00)(32768, 3.184e+00)(65536, 3.205e+00)
};
  \addplot coordinates {
(16, 3.500e+00)(32, 4.250e+00)(64, 3.875e+00)(128, 3.734e+00)(256, 3.816e+00)(512, 3.941e+00)(1024, 4.042e+00)(2048, 4.131e+00)(4096, 4.172e+00)(8192, 4.218e+00)(16384, 4.260e+00)(32768, 4.284e+00)(65536, 4.302e+00)
};
  \addplot+[dashed] coordinates {
(16, 5.562e+00)(32, 6.688e+00)(64, 5.375e+00)(128, 5.570e+00)(256, 5.832e+00)(512, 5.996e+00)(1024, 6.103e+00)(2048, 6.139e+00)(4096, 6.209e+00)(8192, 6.278e+00)(16384, 6.316e+00)(32768, 6.352e+00)(65536, 6.375e+00)
};
\addplot+[dotted] coordinates {
(16, 8.875e+00)(32, 8.531e+00)(64, 8.641e+00)(128, 9.227e+00)(256, 9.180e+00)(512, 9.484e+00)(1024, 9.591e+00)(2048, 9.746e+00)(4096, 9.832e+00)(8192, 9.926e+00)(16384, 9.977e+00)(32768, 1.001e+01)(65536, 1.005e+01)
};
\addplot+[dashed] coordinates {
(16, 1.312e+01)(32, 1.622e+01)(64, 1.459e+01)(128, 1.466e+01)(256, 1.516e+01)(512, 1.561e+01)(1024, 1.591e+01)(2048, 1.611e+01)(4096, 1.625e+01)(8192, 1.637e+01)(16384, 1.644e+01)(32768, 1.649e+01)(65536, 1.655e+01)
};
\addplot+[dotted] coordinates {
(16, 3.112e+01)(32, 2.359e+01)(64, 2.497e+01)(128, 2.617e+01)(256, 2.582e+01)(512, 2.670e+01)(1024, 2.702e+01)(2048, 2.727e+01)(4096, 2.756e+01)(8192, 2.781e+01)(16384, 2.791e+01)(32768, 2.800e+01)(65536, 2.807e+01)
};
\addplot+[dashed] coordinates {
(16, 3.744e+01)(32, 4.194e+01)(64, 4.278e+01)(128, 4.423e+01)(256, 4.541e+01)(512, 4.632e+01)(1024, 4.716e+01)(2048, 4.743e+01)(4096, 4.795e+01)(8192, 4.834e+01)(16384, 4.849e+01)(32768, 4.863e+01)(65536, 4.876e+01)
};
\addplot+[dotted] coordinates {
(16, 6.756e+01)(32, 7.372e+01)(64, 7.647e+01)(128, 7.834e+01)(256, 8.032e+01)(512, 8.182e+01)(1024, 8.307e+01)(2048, 8.399e+01)(4096, 8.496e+01)(8192, 8.528e+01)(16384, 8.566e+01)(32768, 8.596e+01)(65536, 8.615e+01)
};
\end{loglogaxis}
\end{tikzpicture}
\begin{tikzpicture}[baseline=(current axis.south)]%
\begin{loglogaxis}[%
    clip=false,
    scale only axis,
    xmin=14,xmax=75000, ymin=0.1, ymax=1.5,
    title={\textsc{Factors $|\SX|/|\Mirror(I)|$}},
    xtick={16,64,256,1024,4096,16384,65536},
    ytick={0.125, 0.25,0.5,1},
    yticklabels={$2^{-3}$, $2^{-2}$,$2^{-1}$,$2^0$},
    x tick label style={align=center,font=\scriptsize},
    y tick label style={font=\scriptsize},
    xticklabels={$2^4$,$2^6$,$2^8$,$2^{10}$,$2^{12}$,$2^{14}$,$2^{16}$},
    x label style={font=\footnotesize},
    xlabel={Cardinality $|I|$},
    width=0.4\textwidth,
    height=0.25\textwidth,
    cycle list name=MR1LOF
    ]
  \addplot coordinates {
(16, 5.625e-01)(32, 5.859e-01)(64, 6.992e-01)(128, 6.738e-01)(256, 6.836e-01)(512, 7.065e-01)(1024, 7.395e-01)(2048, 7.534e-01)(4096, 7.688e-01)(8192, 7.820e-01)(16384, 7.900e-01)(32768, 7.960e-01)(65536, 8.013e-01)
};
  \addplot coordinates {
(16, 4.375e-01)(32, 5.312e-01)(64, 4.844e-01)(128, 4.668e-01)(256, 4.771e-01)(512, 4.927e-01)(1024, 5.052e-01)(2048, 5.164e-01)(4096, 5.215e-01)(8192, 5.273e-01)(16384, 5.324e-01)(32768, 5.355e-01)(65536, 5.378e-01)
};
  \addplot+[dashed] coordinates {
  (16, 3.477e-01)(32, 4.180e-01)(64, 3.359e-01)(128, 3.481e-01)(256, 3.645e-01)(512, 3.748e-01)(1024, 3.814e-01)(2048, 3.837e-01)(4096, 3.881e-01)(8192, 3.924e-01)(16384, 3.948e-01)(32768, 3.970e-01)(65536, 3.984e-01)
};
\addplot+[dotted] coordinates {
(16, 2.773e-01)(32, 2.666e-01)(64, 2.700e-01)(128, 2.883e-01)(256, 2.869e-01)(512, 2.964e-01)(1024, 2.997e-01)(2048, 3.046e-01)(4096, 3.072e-01)(8192, 3.102e-01)(16384, 3.118e-01)(32768, 3.129e-01)(65536, 3.140e-01)
};
\addplot+[dashed] coordinates {
(16, 2.051e-01)(32, 2.534e-01)(64, 2.280e-01)(128, 2.291e-01)(256, 2.368e-01)(512, 2.439e-01)(1024, 2.486e-01)(2048, 2.518e-01)(4096, 2.539e-01)(8192, 2.558e-01)(16384, 2.568e-01)(32768, 2.576e-01)(65536, 2.586e-01)
};
  \addplot+[dotted] coordinates {
(16, 2.432e-01)(32, 1.843e-01)(64, 1.951e-01)(128, 2.045e-01)(256, 2.017e-01)(512, 2.086e-01)(1024, 2.111e-01)(2048, 2.130e-01)(4096, 2.153e-01)(8192, 2.172e-01)(16384, 2.181e-01)(32768, 2.187e-01)(65536, 2.193e-01)
};
  \addplot+[dashed] coordinates {
(16, 1.462e-01)(32, 1.638e-01)(64, 1.671e-01)(128, 1.728e-01)(256, 1.774e-01)(512, 1.809e-01)(1024, 1.842e-01)(2048, 1.853e-01)(4096, 1.873e-01)(8192, 1.888e-01)(16384, 1.894e-01)(32768, 1.899e-01)(65536, 1.905e-01)
};
  \addplot+[dotted] coordinates {
(16, 1.320e-01)(32, 1.440e-01)(64, 1.494e-01)(128, 1.530e-01)(256, 1.569e-01)(512, 1.598e-01)(1024, 1.623e-01)(2048, 1.640e-01)(4096, 1.659e-01)(8192, 1.666e-01)(16384, 1.673e-01)(32768, 1.679e-01)(65536, 1.683e-01)
};
\end{loglogaxis}
\end{tikzpicture}
\ref{legrandcheb}
\caption{Oversampling factors of spatial discretizations consisting of cosine transformed rank\mbox{-}1 lattices for spans of Chebyshev polynomials with random index sets in $\N_0^{25}$ where the number $d_s$ of active dimensions is fixed.}
\label{fig:numtests_rand_cheb_d25_ds}
\end{figure}

Additionally, we plotted the factors $|\SX|/|\Mirror(I)|$ in order to discuss the advantages of the developed approaches compared to cosine transforming spatial discretizations of $\Pi(\Mirror(I))$. Clearly, a discretization of $\Pi(\Mirror(I))$ has at least a number $|\Mirror(I)|$ of samples. When using lattices for discretization here, the cosine transform leads to approximately half of this number of samples and the resulting sampling set provides a spatial discretization of $\ChebPol(I)$, cf.\ Remark~\ref{rem:remark_cards_cheb_lattices} and Theorem~\ref{thm:spatial_discr_per}. Accordingly, using the detour via spatial discretizations for $\Pi(\Mirror(I))$ we have to expect factors $|\SX|/|\Mirror(I)|$ of at least $1/2$.
Obviously, for the well adjusted approach from Section~\ref{sssec:further_improvement} we observe factors $|\SX|/|\Mirror(I)|$ that are below $1/2$ in the cases where $d_s>3$.
Therefore, the approach developed in this paper seems to be generally more sample efficient than detours via discretizations of trigonometric spaces -- at least when using known rank\mbox{-}1 lattice approaches.

\section*{Conclusion}%

The paper presents approaches for discretizing spaces of Chebyshev polynomials based on cosine transformed rank\mbox{-}1 lattices. The structure of these building blocks allows for an efficient reconstruction algorithm. Moreover, the spatial discretizations are relatively sample efficient, i.e., the numbers of used sampling nodes are efficient -- at least in terms of their complexity. In addition, all construction approaches have acceptable computational complexities. The numerical tests emphasize the differences of the presented improvement steps and indicate that each improvement step, that comes along with growing computational effort, reduces the cardinalities of the resulting spatial discretizations. In most cases, the cardinalities are considerably lower when using the more complex algorithms.

Moreover, it seems that the arising sampling matrices $\ChebMat(\SX,I)$ are well-conditioned, which means that the rows of the matrices form a well-suited frame, cf.~\cite{BaKaPoUl22}. Consequently, the number of samples could be further reduced by applying subsampling ideas from~\cite{BaKaPoUl22}, which would require additional computational effort that is probably disproportionately higher compared to the computational costs of the constructions we developed in this paper.

Anyway, the main goal of this paper was to construct efficient spatial discretizations of spaces of Chebyshev polynomails for use in a so-called dimension-incremental framework similar to that in~\cite{PoVo17} that adaptively builds up spaces of Chebyshev polynomials and uses their spatial discretizations for reconstructing projected polynomial coefficients. Algorithms based on the general ideas presented in this paper have already been used in numerical tests for the treatment of PDEs with random coefficients in~\cite{KaPoTa22}. 

Finally, the approaches developed in this paper can also be used for the spatial discretization of spans of finitely many half-period cosine basis functions. The resulting sets of rank\mbox{-}1 lattices just need to be tent transformed instead of cosine transformed, cf.~\cite{KuoMiNoNu19} for details on the relation of such subspaces of the so-called cosine spaces and spans of Chebyshev polynomials.

\subsection*{Acknowledgments}
The author gratefully acknowledges funding by the Deutsche Forschungsgemeinschaft (DFG, German Research Foundation) – project number 380648269.


{\scriptsize
\begin{thebibliography}{10}

\bibitem{BaKaPoUl22}
F.~Bartel, L.~K\"ammerer, D.~Potts, and T.~Ullrich.
\newblock On the reconstruction of functions from values at subsampled
  quadrature points.
\newblock {\em Math. Comp.}, 93(346):785--809, 2023.

\bibitem{BaNoRi00}
V.~Barthelmann, E.~Novak, and K.~Ritter.
\newblock High dimensional polynomial interpolation on sparse grids.
\newblock {\em Adv. Comput. Math.}, 12:273--288, 2000.

\bibitem{ChCoMiNoTe15}
A.~Chkifa, A.~Cohen, G.~Migliorati, F.~Nobile, and R.~Tempone.
\newblock Discrete least squares polynomial approximation with random
  evaluations - application to parametric and stochastic elliptic {PDE}s.
\newblock {\em ESAIM: M2AN}, 49:815--837, 2015.

\bibitem{ChDeTrWe18}
A.~Chkifa, N.~Dexter, H.~Tran, and C.~G. Webster.
\newblock Polynomial approximation via compressed sensing of high-dimensional
  functions on lower sets.
\newblock {\em Math. Comp.}, 87(311):1415--1450, 2018.

\bibitem{ChIwVo21}
B.~Choi, M.~Iwen, and T.~Volkmer.
\newblock Sparse harmonic transforms {II}: best s-term approximation guarantees
  for bounded orthonormal product bases in sublinear-time.
\newblock {\em Numer. Math.}, 148:293--362, 2021.

\bibitem{CoDaLe13}
A.~Cohen, M.~A. Davenport, and D.~Leviatan.
\newblock On the stability and accuracy of least squares approximations.
\newblock {\em Found. Comput. Math.}, 13:819--834, 2013.

\bibitem{GlMa19}
K.~Glau and M.~Mahlstedt.
\newblock Improved error bound for multivariate {C}hebyshev polynomial
  interpolation.
\newblock {\em Int. J. Comput. Math.}, 96:2302--2314, 2019.

\bibitem{Hoeff63}
W.~Hoeffding.
\newblock Probability inequalities for sums of bounded random variables.
\newblock {\em J. Amer. Statist. Assoc.}, 58:13 -- 30, 1963.

\bibitem{Kae2013}
L.~K{\"a}mmerer.
\newblock Reconstructing multivariate trigonometric polynomials from samples
  along rank-1 lattices.
\newblock In G.~E. Fasshauer and L.~L. Schumaker, editors, {\em Approximation
  Theory XIV: San Antonio 2013}, pages 255--271. Springer International
  Publishing, 2014.

\bibitem{Kae16}
L.~K{\"{a}}mmerer.
\newblock Multiple rank-1 lattices as sampling schemes for multivariate
  trigonometric polynomials.
\newblock {\em J. Fourier Anal. Appl.}, 24:17--44, 2018.

\bibitem{Kae17}
L.~K\"{a}mmerer.
\newblock Constructing spatial discretizations for sparse multivariate
  trigonometric polynomials that allow for a fast discrete {F}ourier transform.
\newblock {\em Appl. Comput. Harmon. Anal.}, 47(3):702--729, 2019.

\bibitem{KaPoTa22}
L.~K\"ammerer, D.~Potts, and F.~Taubert.
\newblock The uniform sparse {FFT} with application to {PDEs} with random
  coefficients.
\newblock {\em Sampl. Theory Signal Proces. Data Anal.}, 20(19), 2022.

\bibitem{KaPoVo17}
L.~K\"{a}mmerer, D.~Potts, and T.~Volkmer.
\newblock High-dimensional sparse {FFT} based on sampling along multiple rank-1
  lattices.
\newblock {\em Appl. Comput. Harmon. Anal.}, 51:225--257, 2021.

\bibitem{KuoMiNoNu19}
F.~Kuo, G.~Migliorati, F.~Nobile, and D.~Nuyens.
\newblock Function integration, reconstruction and approximation using rank-1
  lattices.
\newblock {\em Math. Comp.}, 90(330):1861--1897, 2021.

\bibitem{PoSchmi19}
D.~Potts and M.~Schmischke.
\newblock Approximation of high-dimensional periodic functions with
  {F}ourier-based methods.
\newblock {\em SIAM J. Numer. Anal.}, 59(5):2393--2429, 2021.

\bibitem{PoSc19b}
D.~Potts and M.~Schmischke.
\newblock Learning multivariate functions with low-dimensional structures using
  polynomial bases.
\newblock {\em J. Comput. Appl. Math.}, 403:113821, 2022.

\bibitem{PoVo15}
D.~Potts and T.~Volkmer.
\newblock Fast and exact reconstruction of arbitrary multivariate algebraic
  polynomials in {C}hebyshev form.
\newblock In {\em {11th international conference on Sampling Theory and
  Applications (SampTA 2015)}}, pages 392--396, 2015.

\bibitem{PoVo14}
D.~Potts and T.~Volkmer.
\newblock Sparse high-dimensional {FFT} based on rank-1 lattice sampling.
\newblock {\em Appl. Comput. Harmon. Anal.}, 41(3):713--748, 2016.

\bibitem{PoVo17}
D.~Potts and T.~Volkmer.
\newblock Multivariate sparse {FFT} based on rank-1 {C}hebyshev lattice
  sampling.
\newblock In {\em 2017 International Conference on Sampling Theory and
  Applications (SampTA)}, pages 504--508, 2017.

\bibitem{Po67}
M.~J. Powell.
\newblock On the maximum errors of polynomial approximations defined by
  interpolation and by least squares criteria.
\newblock {\em Comput. J.}, 9:404--407, 1967.

\bibitem{RaAl99}
H.~Rabitz and {\"O}.~F.~Ali\c{s}.
\newblock General foundations of high dimensional model representations.
\newblock {\em J. Math. Chem.}, 25:197--233, 1999.

\bibitem{RaSchw16}
H.~Rauhut and C.~Schwab.
\newblock Compressive sensing {P}etrov--{G}alerkin approximation of
  high-dimensional parametric operator equations.
\newblock {\em Math. Comput.}, 86:661--700, 2017.

\bibitem{SlJo94}
I.~H. Sloan and S.~Joe.
\newblock {\em Lattice methods for multiple integration}.
\newblock Oxford Science Publications. The Clarendon Press Oxford University
  Press, New York, 1994.

\bibitem{SuNuCo14}
G.~Suryanarayana, D.~Nuyens, and R.~Cools.
\newblock Reconstruction and collocation of a class of non-periodic functions
  by sampling along tent-transformed rank-1 lattices.
\newblock {\em J. Fourier Anal. Appl.}, 22(1):187--214, 2016.

\bibitem{ToTr13}
A.~Townsend and L.~N. Trefethen.
\newblock An extension of chebfun to two dimensions.
\newblock {\em SIAM J. Sci. Comp.}, 35:C495--C518, 2013.

\bibitem{Tref13}
L.~N. Trefethen.
\newblock {\em Approximation theory and approximation practice}.
\newblock Society for Industrial and Applied Mathematics (SIAM), Philadelphia,
  PA, 2013.

\end{thebibliography}
}
\end{document}